\documentclass{amsart}

\usepackage{personal_header}

\title{Orbits of Theta Characteristics}

\author[H. W. Braden]{H. W. Braden \orcidA{}}
\address{
School of Mathematics and Maxwell Institute for Mathematical Sciences\\ The University of Edinburgh\\ 
Edinburgh EH9 3FD, Scotland, U.K.
}
\email{hwb@ed.ac.uk}

\author[Linden Disney-Hogg]{Linden Disney-Hogg \orcidB{}}
\address{
School of Mathematics \\ The University of Leeds \\ 
Leeds LS2 9JT, U.K.
}
\email{a.l.disney-hogg@leeds.ac.uk}

\thanks{\textbf{Acknowledgements.} We are grateful to Vanya Cheltsov for discussions. The research of LDH is supported by a UK Engineering and Physical Sciences Research Council (EPSRC) doctoral prize fellowship.}

\thanks{\textbf{Data Availability.} The datasets generated during the current study and the code for their creation/analysis are available from the second author upon reasonable request.}

\begin{document}

\begin{abstract}
The theta characteristics on a Riemann surface are permuted by the induced action of the automorphism group, with the orbit structure being important for the geometry of the curve and associated manifolds. We describe two new methods for advancing the understanding of these orbits, generalising existing results of Kallel \& Sjerve, allowing us to establish the existence of infinitely many curves possessing a unique invariant characteristic as well as determine the number of invariant characteristics for all Hurwitz curves with simple automorphism group. In addition, we compute orbit decompositions for a substantial number of curves with genus $\leq 9$, allowing the identification of where current theoretical understanding falls short and the potential applications of machine learning techniques.
\end{abstract}

\maketitle

\section{Introduction}
Given a smooth, compact, connected Riemann surface $\mathcal{C}$, a \bam{theta characteristic} on $\mathcal{C}$ is a line bundle $L \to \mathcal{C}$ such that $L^2 := L \otimes L$ is isomorphic to the canonical bundle $K_\mathcal{C}$. Given that $\deg K_\mathcal{C} = 2g-2$ is even, where $g = g(\mathcal{C})$ is the genus of the curve, it is necessarily true that theta characteristics exist and moreover that there are exactly $2^{2g}$ of them. The latter follows for, given a theta characteristic $L$ and a 2-torsion degree-0 line bundle $L^\prime \in \Pic^0(\mathcal{C})[2]$, necessarily $L \otimes L^\prime$ is also a theta characteristic. As such the set of characteristics $S(\mathcal{C})$ is an affine space over $\mathbb{Z}_2$ modelled on $H^1(\mathcal{C}, \mathbb{Z}_2)$. Theta characteristics are further distinguished by their \bam{parity}: a theta characteristic $L$ is called odd/even if $\dim H^0(\mathcal{C}, L)$ is odd/even. There are exactly $2^{g-1}(2^g-1)$ odd theta characteristics and $2^{g-1}(2^g+1)$ even theta characteristics \cite{Atiyah1971, Fay1973}. For a general review of theta characteristics and their applications, see \cite{Farkas2012}. 

Theta characteristics are permuted by the action of the automorphism group $\Aut\mathcal{C}$, and in his landmark paper \cite{Atiyah1971} Atiyah showed that every $f \in \Aut(\mathcal{C})$ leaves at least one characteristic invariant. Kallel and Sjerve \cite{Kallel2010} then greatly expanded upon this result, showing how to compute the orbits of characteristics given knowledge of the rational representation of their automorphism group, and using that to give further results about the number of characteristics invariant under a given automorphism. In \S\ref{sec: action of rational rep}-\ref{sec: group cohomology} we will briefly review their results, as we will then use this to give a group cohomology description of invariant characteristics. This allows us to generalise the results of \cite{Kallel2010}, giving the two main new results of this paper
\begin{proposition}
If there exists $f \in G \leq \Aut(\mathcal{C})$ such that $f$ has odd order, $\pangle{f}$ is subnormal in $G$, and $g(\mathcal{C}/\pangle{f}) = 0$, then $\mathcal{C}$ has a unique theta characteristic invariant under the action of $G$.
\end{proposition}

\begin{theorem}
    There are infinitely many curves, both non-hyperelliptic and hyperelliptic, with a unique invariant characteristic. 
\end{theorem}

In \S\ref{sec: computations of orbit tables} we will describe tables of curves of genus $\leq 9$ and their orbit decompositions, stratified by the automorphism group action of the curve, given in Appendix \ref{sec: tables of orbits}. This will serve two purposes: first, it will act as a collation of plane curve models of Riemann surfaces with automorphisms; and second, it will identify cases where the current theory fails to explain why a curve has a unique invariant characteristic. Moreover, by synthesising the computational tools in Sage with those of \cite{Behn2013} we will demonstrate the applicability of machine learning tools in gaining insight. 

In particular, this machine learning will lead us to the work in \S\ref{sec: Hurwitz curves and Dolgachev} where, by extending work of Dolgachev on invariant bundles over modular curves \cite{Dolgachev1997a}, we will calculate the number of invariant characteristics for all Hurwitz curves with simple automorphism group. These we can compare with the predictions of the machine learning method.  

Part of this work was completed during the PhD thesis of LDH \cite{DisneyHogg2023}. The code that was used there is available at \url{https://github.com/DisneyHogg/Riemann_Surfaces_and_Monopoles}.

\section{Action of the Rational Representation}\label{sec: action of rational rep}

Choosing a basis $\pbrace{a_i, b_i}_{i=1}^g$ of the homology group $H_1(\mathcal{C}, \mathbb{Z})$ with canonical intersection matrix $J = \begin{psmallmatrix} 0 & I_g \\ -I_g & 0 \end{psmallmatrix}$, i.e. 
\[
a_i \circ b_j = \delta_{ij}, \quad a_i \circ a_j = 0 = b_i \circ b_j,
\]
yields the \bam{rational representation} $\rho_r : \Aut(\mathcal{C}) \to \Sp_{2g}(\mathbb{Z})$. The starting point for this paper is the work of \cite{Kallel2010}, where the authors use the identification between theta characteristic and spin structures on a Riemann surface to provide a particular choice of isomorphism $S(\mathcal{C}) \cong H_1(\mathcal{C}, \mathbb{Z}_2) \cong \mathbb{Z}^{2g}$ such that ${x} \in \mathbb{Z}_2^{2g}$ transforms under $f \in \Aut(\mathcal{C})$ as 
\begin{equation}\label{eq: transform of binary vectors}
{x} \mapsto R^T {x} + {v} {\mod 2},
\end{equation}
where $R = \rho_r(f)$ and the vector ${v}$ is computed as $v_i = \sum_{j<j^\prime} R_{ji} R_{j^\prime i} J_{j j^\prime}$. One can think of this action on ${x}$ as matrix multiplication 
\[
\begin{pmatrix}
    {x} \\ 1 
\end{pmatrix} \mapsto \begin{pmatrix}
    R^T & {v} \\ 0 & 1 
\end{pmatrix} \begin{pmatrix}
    {x} \\ 1 
\end{pmatrix}.
\]

\begin{example}
    Suppose $R=I$ is the identity matrix, then 
    \[
v_i = \sum_{j < j^\prime} \delta_{ji} \delta_{j^\prime i} J_{j j^\prime} = 0 . 
    \]
    As such, as a sense check we can see that any theta characteristic is invariant under the identity transformation.
\end{example}

This affine action is a classical result (see, for example, \cite[\S{V.1}]{Igusa1972}), but one of the strengths of Kallel and Sjerve's approach using spin structures is that the following result is not difficult to prove. 

\begin{lemma}[\cite{Johnson1980}, \S{5}]\label{lemma: parity of spin structure}
    Writing ${x} = ({u}, {v})$ for ${u}, {v} \in \mathbb{Z}_2^g$, the parity of the associated theta characteristic is $q({x}) := {u} \cdot {v}$. 
\end{lemma}

As a result of Lemma \ref{lemma: parity of spin structure} we see the parity of a spin structure is given by a quadratic form on $H_1(\mathcal{C}, \mathbb{Z}_2)$ such that the associated bilinear form $H({x}, {y}) := q({x} + {y}) - q({x}) - q({y})$ is the reduction mod 2 of the intersection pairing on $H_1(\mathcal{C}, \mathbb{Z})$. 

Given explicitly the rational representation, Equation (\ref{eq: transform of binary vectors}) and Lemma \ref{lemma: parity of spin structure} provide a fast and exact method to compute the orbit decomposition of theta characteristics (as opposed to slow and numerical methods using the analytical representation and the Abel-Jacobi map, see \cite{DisneyHogg2023, DisneyHogg2022b}), split by parity. An implementation of this method may be seen in the code from \cite{DisneyHogg2022b} available at \url{https://github.com/DisneyHogg/Brings\_Curve}, and this is how we shall compute tables of orbit decomposition in \S\ref{sec: computations of orbit tables}.

We shall at this stage introduce another concept which shall be used later, that of the \bam{signature} of the rational representation. 
\begin{definition}[\cite{Magaard2002}, \S2]\label{def: signature and ramification type}
	Take $G \leq \Aut(\mathcal{C})$ acting on a curve and denote by $Q_i$, $i=1, \dots, r$, the branch points of the quotient map $\pi : \mathcal{C} \to \mathcal{C}/G$. Denote by $C_i$ the conjugacy class of $\Stab(P_i)$ for any $P_i \in \pi^{-1}(Q_i)$, and let $c_i$ denote the order of an element of $C_i$. As the $Q_i$ are branch points, the $C_i$ are nontrivial conjugacy classes, and $c_i >1$. We call the \bam{ramification type} of the $G$ action on $\mathcal{C}$ the data of the tuple $(g(\mathcal{C}), G, (C_1, \dots, C_r))$. The vector $\bm{c} = (c_1, \dots, c_r)$ (or sometimes $(g_0; c_1, \dots, c_r)$ where $g_0 = g(\mathcal{C}/G)$ is the \bam{orbit genus}) is called the \bam{signature}. We will use exponents to indicate how many times a value of $c_i$ is repeated as in \cite{Broughton1991}, e.g. the signature $(0; 2, 2, 2, 3)$ is written as $(0; 2^3, 3)$. A \bam{generating vector} for the action is $\pbrace{\alpha_1, \beta_1, \dots, \alpha_{g_0}, \beta_{g_0}, \gamma_1, \dots, \gamma_r} \subset G$ such that 
	\[
	\gamma_1^{c_1} = \dots = \gamma_r^{c_r} = \prod_{i=1}^{g_0} (\alpha_i \beta_i \alpha_i^{-1} \beta_i^{-1}) \prod_{j=1}^r \gamma_j = 1.
	\]
\end{definition}

The Riemann-Hurwitz theorem immediately says that 
\begin{equation}\label{eq: RH for group quotients}
2g-2 = \abs{G} \psquare{(2g_0-2) + \sum_i \pround{1 - \frac{1}{c_i}}}. 
\end{equation}

The (character of the) rational representation completely determines the corresponding signature \cite[p.~401]{Ries1993}. Conversely, given a signature satisfying Equation (\ref{eq: RH for group quotients}) and a corresponding generating vector in an abstract group $G$, there exists a Riemann surface with associated $G$ action, and the generating vector determines the (character of the) rational representation via the Eichler trace formula \cite{Breuer2000}.  

\section{Invariant Characteristics and Group Cohomology}\label{sec: group cohomology}

Suppose we had parametrised our set of theta characteristics differently in terms of ${x}^\prime \in \mathbb{Z}_2^{2g}$ where ${x} = {x}^\prime + {y}$ for some fixed ${y} \in \mathbb{Z}_2^{2g}$. This is reparametrisation by translation. One can derive the corresponding action on ${x}^\prime$ by noting 
\begin{align*}
    {x} = {x}^\prime + {y} &\mapsto R^T({x}^\prime + {y}) + {v}, \\
    & = R^T{x}^\prime + \psquare{{v} + (R^T - I){y}} + {y},
\end{align*}
and so
\[
{x}^\prime \mapsto R^T {x}^\prime + {v}_y,
\]
where ${v}_y = {v} + (R^T-I) {y}$. If ${y}$ were fixed by the action (\ref{eq: transform of binary vectors}) then ${v}_y=0$ and moreover the converse is true. As such, we have seen the following proposition. 

\begin{prop}\label{prop: invariant characteristic exists iff affine rep is equivalent to linear}
    There is an invariant characteristic on a curve $\mathcal{C}$ if and only if the corresponding affine representation on $\mathbb{Z}_2^{2 g}$
    \[
\begin{pmatrix}
    {x} \\ 1 
\end{pmatrix} \mapsto \begin{pmatrix}
    R^T & {v} \\ 0 & 1 
\end{pmatrix} \begin{pmatrix}
    {x} \\ 1 
\end{pmatrix}
\]
is equivalent by a translation ${x} \to {x}^\prime := {x} - {y}$ to the linear action 
    \[
\begin{pmatrix}
    {x}^\prime \\ 1 
\end{pmatrix} \mapsto \begin{pmatrix}
    R^T & 0 \\ 0 & 1 
\end{pmatrix} \begin{pmatrix}
    {x}^\prime \\ 1 
\end{pmatrix}
\]
for some ${y} \in \mathbb{Z}_2^{2 g}$. 
\end{prop}

To utilize Proposition \ref{prop: invariant characteristic exists iff affine rep is equivalent to linear}, we note that the question of whether an affine representation can be reduced to a linear one may be presented as a cohomology problem\footnote{LDH is grateful to Andrew Beckett for highlighting this to him.}. In particular, given a group $G$ and a (left) $G$-representation $\rho : G \to \GL(V)$ we have the following results.
\begin{prop}\label{prop: affine representations are cocycles}
An affine representation of $G$ on $V$ which acts multiplicatively via $\rho$ determines a 1-cocycle in the group cohomology\footnote{For an introduction to group cohomology with necessary definitions see \cite[\S6]{Weibel1995}. We shall from here on in drop the $\operatorname{Grp}$ subscript as through context it shall not cause confusion with any other cohomologies in this work.} $H^1_{\operatorname{Grp}}(G, V)$ (making $V$ into a $G$-module in the natural way with $\rho$) with the standard linear representation $G \times V \to V$, $(g, x) = \rho(g)x$, corresponding to the zero 1-cocycle. 
\end{prop}
\begin{proof}
An affine representation acting multiplicatively by $\rho$ is defined by 
\begin{align*}
G \times V &\to V, \\
(g, x) &\mapsto g \cdot x := \rho(g)x + v(g).
\end{align*}
for some set map $v : G \to V$. By definition the set map $v$ is exactly a 1-cochain in group cohomology, with the linear representation giving the zero 1-cochain. Moreover, to truly get an action we require $\forall \, g, h \in G, \, x \in V$, $g\cdot(h \cdot x) = (gh) \cdot x$. We can compute from the definition 
\begin{align*}
g \cdot (h \cdot x) &= \rho(g)(h \cdot x) + v(g) , \\ 
&= \rho(g)[\rho(h)x + v(h)] + v(g) , \\
&=\rho(gh)x + \rho(g) v(h) + v(g), \\
&=(gh)\cdot x + \rho(g) v(h) - v({gh}) + v(g),
\end{align*}
and so we must have 
\begin{equation}\label{eq: affine representation closed condition}
\forall \, g, h \in G, \, 0=\rho(g) v(h) - v({gh}) + v(g). 
\end{equation}
In particular, setting $g=e$ the identity element in Equation (\ref{eq: affine representation closed condition}) shows $v(e)=0$. Equation (\ref{eq: affine representation closed condition}) is exactly the condition that the 1-cochain $v$ is in fact a 1-cocycle \cite[Example 6.5.6]{Weibel1995}.
\end{proof}

\begin{prop}\label{prop: equivalent affine reps differe by coboundaries}
    Two affine representations as defined in Proposition \ref{prop: affine representations are cocycles} are equivalent under a translation of $V$ if and only if the associated 1-cocycle is a 1-coboundary. 
\end{prop}
\begin{proof}
Fixing $y \in V$ and $v : G \to V$ defining an affine representation we have that 
\begin{align*}
g \cdot (x+y) &= \rho(g)(x+y) + v(g) , \\
&= \pbrace{\rho(g)x + \psquare{v(g) + (\rho(g)-I)y}} + y.
\end{align*} 
This defines a different affine action on $V$ given by 1-cocycle $v_y(g) := v(g) + (\rho(g)-I)y$. This new affine action is actually linear if and only if 
\begin{equation}\label{eq: affine representation boundary condition}
\forall g \in G, \, (\rho(g)-I)y + v(g) = 0 \Leftrightarrow \forall g \in G, \, g \cdot y = y \Leftrightarrow y \in V^G, 
\end{equation}
where we have used $V^G$ to denote the subset of $V$ invariant under $G$. The condition that $v(g) = \rho(g)y - y$ for some $y \in V$ is exactly the condition that $v$ is a 1-coboundary \cite[Example 6.5.6]{Weibel1995}. 
\end{proof}

In the case at hand of considering the group action of the automorphism group on theta characteristics the representation will be the reduction mod 2 (with the mod 2 reduction of $R$ denoted by $\overline{R}$, following \cite{Kallel2010}) of the transpose of the rational representation $\rho = \overline{\rho}_r^T$ acting on $V = H_1(\mathcal{C}, \mathbb{Z}_2) \cong \mathbb{Z}_2^{2 g}$. Moreover, we can count the number of invariant characteristics as the size of $H^0(G, V)$, as $H^0(G, V) = V^G$ is exactly the submodule of invariants. This fact gives us an immediate refinement of \cite[Corollary 1.3]{Kallel2010}.
\begin{prop}\label{prop: number of invariant characteristics}
    The number of characteristics invariant under the action of the whole group is either 0 or $2^k$, where $k=\dim H^0(G, V)$ is the dimension\footnote{When writing $\dim$ for the dimension of the group cohomology of a $\mathbb{Z}_2$ vector space, we naturally mean the dimension over $\mathbb{Z}_2$.} of the subspace of invariants. 
\end{prop}
\begin{proof}
    This is immediate from the fact $H^0(G, V)$ is a vector space over $\mathbb{Z}_2$. 
\end{proof}

We shall now want to consider two simple examples, for which we require the (proof of the) following lemma.
\begin{lemma}[\cite{Atiyah1971}, Lemma 5.1]\label{lemma: affine transformation fixing quadratic has fixed point}
	Let $V$ be a finite-dimensional $\mathbb{Z}_2$ vector space and $q : V \to \mathbb{Z}_2$ a quadratic function fixed under an affine transformation $x \mapsto Ax+b$ whose associated bilinear $H$ defined by 
	\[
	H(x, y) = q(x+y) - q(x) - q(y)
	\]
	is non-degenerate. Then the affine transformation has a fixed point.
\end{lemma}
\begin{proof} We shall recall the proof of \cite{Atiyah1971}. As the transform preserves $q$ we get 
	\[
	q(x) = q(Ax+b) = q(Ax) + q(b) + H(Ax, b).
	\]
	Setting $x=0$ gives $q(b) = 0$ and hence $q(x) = q(Ax) + H(Ax, b)$. We can thus say 
	\begin{align*}
		H(x, y) &= q(x+y) - q(x) - q(y) , \\
		&= \psquare{q(A(x+y)) + H(A(x+y), b)} - \psquare{q(Ax) + H(Ax, b)} 
  \\ &\phantom{=} \quad
           - \psquare{q(Ay) + H(Ay, b)}, \\
		&= \psquare{q(Ax + Ay) - q(Ax) - q(Ay)} , \\
		&= H(Ax, Ay) ,
	\end{align*}
	and so defining $A^\ast$ to be the dual of $A$ with respect to the non-degenerate inner product $H$ we have $A^\ast A = I$. Suppose we have $x \in \Ker (A-I)^\ast$, then 
	\begin{align*}
		A^\ast x = x \Rightarrow Ax = x \Rightarrow H(x,b) = 0
	\end{align*}
	and hence we know $b \perp \Ker (A-I)^\ast$. Now $(\Ker (A-I)^\ast)^\perp = \im(A-I)$, and so 
	\[
	b \in \im(A-I) \Rightarrow \exists y \in V, \, b = (A-I)y \Rightarrow \exists y \in V, \, Ay+b = y. 
	\]
\end{proof}

We are now ready to consider two examples. 

\begin{example}[Hyperelliptic Involution]\label{ex: hyperelliptic involution fixes all characteristics 2}
Suppose we just have $G = C_2 = \pbrace{\pm 1}$, where the generator of the $C_2$ is the hyperelliptic involution, for which the rational representation is given by $R = -I$. For a cochain $v : G \to V$ to be closed it must satisfy $v(1) = 0 \in V$, and then $v(-1) := x$ is arbitrary. It is a coboundary if in addition there exists $ y \in V$, $v(-1) = y - y = 0$, and so with the hyperelliptic representation the cohomology group is given by $H^1(C_2, V) = V$. Note this is what we would expect: if the representation is trivial, then for all ${y}\in V$ we have ${v}_y = {v}$, meaning that the action is always given by the specific shift ${v}$. We may now use Lemma \ref{lemma: affine transformation fixing quadratic has fixed point}. The quadratic function on $H_1(\mathcal{C}, \mathbb{Z}_2)$ given by the parity is preserved under the affine action of $\Aut(\mathcal{C})$, so the affine transformation given by the hyperelliptic involution has a fixed point. The vector $b$ in the proof of Lemma \ref{lemma: affine transformation fixing quadratic has fixed point} is the value $v(-1)$, and so we must have $v(-1) \in \im 0 = 0$. Hence we have an invariant spin structure, and moreover we have $2^{2 g} = \abs{V}$ of them. 
\end{example}

\begin{example}[Cyclic groups]\label{ex: group cohomology cyclic group}
Let us now try and understand \cite[Theorem 1.1]{Kallel2010} in this language of group cohomology. Suppose we take a cyclic automorphism group $\pangle{f} = G$, and let $n$ be the order of $f$. A closed cochain must have $v(1) =0$ as before, and then it is specified by $v(f)$, which is subject only to the condition that 
\begin{equation}\label{eq: constraint of cocycle image}
(I + A + \dots + A^{n-1})v(f) = v(f^n) = v(1) = 0 , 
\end{equation}
where $A = \bar{R}^T = \rho(f)$ (already reduced mod 2). Provided $A \neq I$, given that we have $(A-I)(\sum_{k=0}^{n-1} A^k) = A^n-I=0$, the sum $\sum_{k=0}^{n-1} A^k$ has nontrivial kernel in which some $v(f)\ne0$ must lie. A coboundary is given by $v(f) = (A-I)y$ for some $y \in V$, and hence 
$$
H^1(G, V) \cong \Ker (I + \dots + A^{n-1})/\operatorname{Im}(A-I).
$$ 
Note this result is contained in \cite[Theorem 6.2.2]{Weibel1995}.

Certainly if $A-I$ is invertible, then there is necessarily an invariant spin structure as $H^1 = 0$, and it is unique as then $H^0(G, V) = \Ker (A-I) = 0$. Moreover the converse is true that if there is a unique invariant characteristic then $A-I$ is invertible, as $\Ker (A-I)=0$ is exactly the condition for invertibility. By the argument required in \cite[p.~17]{Braden2012}, we can see that for $A-I$ to be invertible, it is necessary that $\sum_{k=0}^{n-1}A^k=0$, and in fact this is sufficient when $n$ is odd as we can simply write down the inverse which is $n^{-1} \sum_{k=1}^{n} kA^{k-1}$. Note also certainly if $n=2$ then $A-I$ cannot be invertible as 
\[
(A-I)^2 = A^2 - 2A + I = 0.
\]
This in fact generalises; if $n=2^k$ then $A-I$ cannot be invertible as 
\[
(A-I)^{n} = \sum_{i=0}^{n} \binom{n}{i} A^i = A^n - I = 0
\]
using Lucas' theorem for the binomial coefficient of a prime power mod that prime (see for example \cite[Theorem 3]{Fine1947}, or alternatively prove this with induction).

Now, as in Example \ref{ex: hyperelliptic involution fixes all characteristics 2}, the quadratic function given by the parity is preserved under the affine transformation generating the group action on $H_1(\mathcal{C}, \mathbb{Z}_2)$, so the transformation has a fixed point. As such, the remaining question is just how many invariant characteristics there are. This is given by $\dim \ker (A-I)$ and Proposition \ref{prop: number of invariant characteristics}, and \cite{Kallel2010} shows how to compute whether this is zero using cyclotomic polynomials. Importantly they show that a necessary and sufficient condition when $n$ is odd is that the genus of the quotient $g = g(\mathcal{C}/\pangle{f})$, related to $f$ by $\operatorname{rank}_{\mathbb{Z}} \ker(\rho_r(f)-I) = 2 g$, is 0. By combining the work of \cite{Ries1992}, we see that when $n$ is even the necessary and sufficient conditions are that $g(\mathcal{C}/\pangle{f^{2^l}}) = 0$ for all $l$ such that $2^l | n$. 
\end{example}
\begin{remark}
    We can also use \cite[Theorem 1]{Scott1977} to say more. Namely, suppose that $G$ is generated by elements $\gamma_1, \dots, \gamma_r$ such that $\prod_i \gamma_i = 1$, then  
    \begin{align*}
        \sum_{i=1}^r\psquare{2g - \dim H^0(\pangle{\gamma_i}, V)} &\geq \psquare{2g - \dim H^0(G, V)} + \psquare{2g - \dim H^0(G, V^\ast)}, \\
        \Rightarrow  \dim H^0(G, V) & \geq \sum_i \dim H^0(\pangle{\gamma_i}, V) - 2g(r-2) - \dim H^0(G, V^\ast),
    \end{align*}
    where $V^\ast = \Hom(V, \mathbb{Z}_2)$ is the dual $G$-module. Scott defines the \bam{homology invariant} $h$ to be the nonnegative integer which is the deficit in the inequality, as well as $\tilde{H}^1(G, V)$ to be the subgroup of $H^1(G, V)$ which restricts to $0 \in H^1(\pangle{\gamma_i}, V)$ for each $i$, showing $h \geq \dim \tilde{H}^1(G, V)$. As such, we can use knowledge of the number of characteristics invariant under the $\pangle{\gamma_i}$ to give information about how many possible invariant characteristics there may be.
\end{remark}

A key question we shall want to address using group cohomology is when the action of $G$ has a \bam{Unique Invariant Characteristic (UIC)}. To make more progress, we use the inflation-restriction exact sequence. That is for normal subgroup $N \triangleleft G$ and abelian group $V$ with $G$ action we have \cite[p.~196]{Weibel1995}
\[
0 \to H^1(G/N, V^N) \to H^1(G, V) \to H^1(N, V)^{G/N} \to H^2(G/N, V^N) \to H^2(G, V) , 
\]
so named because the 3 inner maps are \textbf{inflation}, \textbf{restriction}, and \textbf{transgression}. Suppose we know $H^1(N, V) = 0$ and $H^0(N, V) = V^N = 0$ (as we have when $N$ is an odd-order cyclic group quotienting to $\mathbb{P}^1$), then denoting $K = G/N$ we have 
\[
0 \to H^1(K, 0) \to H^1(G, V) \to 0 \to H^2(K, 0) \to H^2(G, V) . 
\]
This trivially gives $H^1(G, V) = 0$, and moreover because we have $(V^N)^{(G/N)} \cong V^G$, we get $H^0(G, V)=0$. 

In fact we do not need the restrictive condition that $H^1(N, V)=0$ in the situation we care about, as we want the case where there is a unique characteristic invariant under $N$, which is the case when $H^0(N, V) = V^N =0$ but also when the specific affine representation of $G$ as an element in $H^1(G, V)$ is the zero class when restricted to the action as an element in $H^1(N, V)$. This means that the class in $H^1(G, V)$ is in the kernel of the restriction map, and as we have said that the image of the inflation map is trivial, this means the class in $H^1(G, V)$ must be the zero class.  
Read together these tell us that if we have a normal subgroup given with a UIC, then the action of the whole group has a UIC. Moreover, this extends to if we have a subgroup which is subnormal in the original group (that is $H \leq G$ such that there exist $ H_i$ with $H \triangleleft H_1 \triangleleft \dots \triangleleft H_k \triangleleft G$). In summation this gives the following proposition. 

\begin{prop}\label{prop: SOC curves have UIC}
    If there exists $f \in G \leq \Aut(\mathcal{C})$ such that $f$ has odd order, $\pangle{f}$ is subnormal in $G$, and $g(\mathcal{C}/\pangle{f}) = 0$, then $\mathcal{C}$ has a unique theta characteristic invariant under the action of $G$. We will call a curve with this property Subnormal Odd Cyclic (SOC). 
\end{prop}

This condition captures the existence of unique invariant characteristics in the case of even-order cyclic group actions seen earlier in Example \ref{ex: group cohomology cyclic group}. In that case, writing $n = 2^k m$ for some odd integer $m$, the necessary and sufficient conditions previously described are equivalent to the subnormal subgroup $C_m$ having quotient genus 0. In the case of SOC curves, we are also able to determine the parity of the invariant characteristic from the signature.
\begin{prop}\label{prop: parity of SOC UIC}
	Given a SOC curve when the subnormal cyclic group $\pangle{f}$ has odd order $n$ and signature $(0; c_1, \dots, c_r)$, the parity is 
	\[
	n \sum_{i=1}^r \frac{c_i - c_i^{-1}}{8} \mod 2.
	\]
\end{prop}
\begin{proof}
Note that the $c_i$ must be odd because they divide $n$. Hence, using \cite[Equation 16]{Serre1990} we have that the parity is
\begin{align*}
    \sum_{P \in \mathcal{C}} \frac{\abs{\operatorname{Stab}(P)}^2-1}{8} &= \sum_{Q \in \mathcal{C}/G} \sum_{\substack{{P \in \mathcal{C}} \\ {\pi(P) = Q}}} \frac{\abs{\operatorname{Stab}(P)}^2-1}{8} \mod 2, \\
    &= n  \sum_{\pi(P) \in \mathcal{C}/G}  \frac{\abs{\operatorname{Stab}(P)}-\abs{\operatorname{Stab}(P)}^{-1}}{8} \mod 2, \\
    &= n \sum_{i=1}^r \frac{c_i - c_i^{-1}}{8} \mod 2. 
\end{align*}
\end{proof}

\section{Computation of Orbit Tables}\label{sec: computations of orbit tables}

In Appendix \ref{sec: tables of orbits} we shall provide tables of orbit decompositions for many different curves of genera $\leq 9$, giving those for all possible curves of genus 2, 3, and 4. In this section we shall make some comments on the computation of these tables and their results. 

The tables are computed using the representation in terms of binary vectors via spin structures as in \cite{Kallel2010}· We restrict to curves for which we can represent the curve in plane form as $f(x,y) = 0$, whereby we may use the method of \cite{Bruin2019} implemented in Sage to compute the rational representation. We will also want to compute the signature of the action where possible, defined as follows. Throughout signatures are computed with the help of the LMFDB \cite{LMFDB}, and in cases of ambiguity were verified by comparing the character of the rational representation found using this signature and the Eichler trace formula \cite[p.~41]{Breuer2000} to that found by computing directly with Sage, or by determining the signature from the rational representation as in \cite{Ries1993}.

In doing the orbit decomposition computations, we are aided by the fact that we need only choose one representative curve from each class because the decomposition is uniquely determined by the rational representation and two curves in the interior of an equisymmetric family (in the sense of \cite{Broughton1990}) have equivalent rational representations \cite[p.~896]{ReyesCarocca2022}. One can understand this intuitively: if the coefficients of the curve are varied only slightly (and generically) then as the rational representation is given in terms of integer valued matrices these would not be expected to change.

It is shown in \cite{Kallel2010, Biswas2007} that the hyperelliptic involution (if it exists) is the only nonidentity automorphism $\iota$ for which $\rho_r(\iota) = I \mod 2$, and so is the only nonidentity automorphism of a curve that fixes every theta characteristic, hence we know that the \bam{reduced automorphism group} \cite{Rauch1970, Popp1972}
\[
\Autb(\mathcal{C}) := \left \lbrace \begin{array}{cc}
    \Aut(\mathcal{C})/\pangle{\iota}, & \mathcal{C} \text{ is hyperelliptic}, \\
    \Aut(\mathcal{C}), & \text{otherwise},
\end{array} \right.
\]
acts faithfully on the theta characteristics. There is no a priori reason to expect $\Autb(\mathcal{C})$ to have an action on $\mathcal{C}$ when $\mathcal{C}$ is hyperelliptic \cite{Rauch1970}. 

As such we shall give tables of curves of a given genus as a plane curve $f(x,y)=0$, their reduced automorphism group (with the GAP group ID if the presentation given of the group is not specific \cite{GAP4}), the quotient genus $g_0$ and signature $\bm{c}$ of the $\Autb$ action on the curve (recall Definition \ref{def: signature and ramification type}) when it exists and is known, the $\Autb$ orbits of the odd and even characteristics respectively presented as a list of values $a_b$ indicating $b$ orbits of size $a$, and the total number of characteristics $I$ invariant under the group action. When giving $f$ we shall leave in free parameters where possible, not specifying values that must be avoided. The code to recreate this data (except for the signature) is given in the Sage notebooks \path{list_of_plane_curves.ipynb} and \path{theta_characteristic_orbit.ipynb}. Table \ref{tab: orbit decomposition, genus 2} shows the table constructed in the case of genus-2 curves. 

\begin{remark}
    We will use the convention that the dihedral group $D_n$ be of size $2n$.
\end{remark}

\begin{longtable}
{p{0.37\linewidth}|p{0.16\linewidth}|p{0.16\linewidth}|p{0.16\linewidth}|p{0.04\linewidth}}
\caption{Orbit decomposition, all genus-2 curves} \\
    $f$ & $\Autb$, $\bm{c}$ & Odd & Even & I\\ \hline \hline 
    $y^2 - (x^2-1)(x^2-a)(x^2-b)$ & $C_2$, $(1; 2^2)$ &  $2_3$ & $1_4, 2_3$ & 4 \\
    $y^2 - (x^2-1)(x^2-a)(x^2-a^{-1})$ & $V_4$, $(0; 2^5)$ &  $2_1, 4_1$ & $1_2, 2_2, 4_1$ & 2 \\
    $y^2 - (x^5 - 1)$ & $C_5$, $(0; 5^3)$ & $1_1, 5_1$ & $5_2$ & 1 \\
    $y^2 - (x^6 - ax^3 + 1)$ & $S_3$, $(0; 2^2, 3^2)$ & $6_1$ & $1_1, 3_3$ & 1 \\
    $y^2 - (x^6 - 1)$ & $D_{6}$, $(0; 2^3, 3)$ & $6_1$ & $1_1, 3_1, 6_1$ & 1 \\
    $y^2 - x(x^4 - 1)$ & $S_4$ & $6_1$ & $4_1, 6_1$ & 0 
\label{tab: orbit decomposition, genus 2}
\end{longtable}

\begin{remark}
    The code provided may easily be adapted to compute the orbit decomposition for all subgroups of the automorphism group. 
\end{remark}

Every genus-2 curve with a unique invariant characteristic satisfied the conditions of Proposition \ref{prop: SOC curves have UIC}, that is it is SOC. 

We shall now make some remarks about Tables \ref{tab: orbit decomposition, genus 3}, \ref{tab: orbit decomposition, genus 4}, \ref{tab: orbit decomposition, genus 5}, \ref{tab: orbit decomposition, genus 6}, \ref{tab: orbit decomposition, genus 7}, \ref{tab: orbit decomposition, genus 8}, and \ref{tab: orbit decomposition, genus 9}.
\begin{itemize}
    \item At genus 3, not every curve with a unique invariant characteristic is SOC. There is a single exception, Klein's curve, whose automorphism group $\PSL_3(\mathbb{F}_2)$ is simple and so cannot have a nontrivial subnormal cyclic group. There is a $C_7$ subgroup of the automorphism group quotienting to $\mathbb{P}^1$ (as clearly seen by writing Klein's curve in Lefschetz form \cite{Lefschetz1921, Lefschetz1921a, Zomorrodian2010}), but it is not subnormal. 
    \item The non-hyperelliptic genus-3 curve with $\Autb = S_4$ is the first example of a curve with large automorphism group without $I \leq 1$. 
    \item At genus 4 one sees the first example of curves of the same genus with the same pair $(\Autb, \bm{c})$, but different orbit decompositions. This highlights the fact that the signature does not fully the determine the rational representation, one must also pick a generating vector \cite[Definition 2.2]{Broughton1991}. The converse, that the (character of the) rational representation determines the signature, is true \cite{Ries1993}. 
    \item At genus 4 not all curves with a unique invariant characteristic are SOC. The exceptions are 
    \begin{enumerate}
        \item the curve with $(\Autb, \bm{c}) = (A_4, (0; 2, 3^3))$,
        \item the curve with $(\Autb, \bm{c}) = (S_5, (0; 2, 4, 5))$, Bring's curve.
    \end{enumerate}
    More about the orbit decomposition on Bring's curve is said in \cite{DisneyHogg2022b}; here we only note that similarly to Klein's curve there is an odd order cyclic group quotienting to $\mathbb{P}^1$ (here a $C_5$) that is not subnormal. On the $A_4$ curve, we note that the quotient of the curve by the $C_3$ action has genus 1, and hence presents the first case where the existence of a UIC is not clearly governed by an odd cyclic quotient to $\mathbb{P}^1$. 
    \item At genus 5 the Wiman octic is not SOC; there is a unique characteristic invariant under the $((C_4 \times C_2) \rtimes C_4) \rtimes C_3$ normal subgroup, but not under any subnormal cyclic group. Moreover, again we find that the quotient by $C_3$ has genus 1, and so the Wiman octic presents the second case where the existence of a UIC is not clearly governed by an odd cyclic quotient to $\mathbb{P}^1$. The curve with automorphism group $C_3 \times D_5$ is clearly SOC. 
    \item At genus 6 every curve written down with a UIC is SOC.
    \item Edge describes the orbits of some of the odd characteristics of the genus-7 Fricke-Macbeath curve in terms of tangent hyperplanes \cite{Edge1967}. 
    \item All the genus-7 curves written with a unique invariant characteristic are SOC.
    \item At genus 8 the two curves which have a UIC are SOC. 
    \item Not all the genus-9 curves written here with a UIC are SOC; the hyperelliptic curve with $\Autb \cong A_5$ cannot have a subnormal cyclic group. 
\end{itemize}

At genera greater than 8 the computations were becoming prohibitively slow, with the calculation of the symplectic automorphism group of the Fricke octavic curve in Sage taking just under three hours on a Intel Core i5-8350U CPU at 1.70GHz. Leaving behind the criteria of requiring a plane model of the curve, one can compute additional examples of theta characteristic decompositions using the code from \cite{Behn2013}, available at \url{https://github.com/rojas-ani/sage-routines},\footnote{LDH is grateful to Anita Rojas for her correspondence on the workings of this code.} from the data of a group, signature, and choice of generating vector provided the quotient genus is 0. The Sage notebook \path{genus_order_invariants_data.ipynb} shows how this can be done.

Having computed now many examples in low genera, we can pick out some families of curves with unique invariant characteristics, giving us the following theorem. 
\begin{theorem}\label{thm: infinitely many UIC curves}
    There are infinitely many curves, both non-hyperelliptic and hyperelliptic, with a unique invariant characteristic. 
\end{theorem}
\begin{proof}
    It is sufficient to consider only curves of Lefschetz type, in particular the two families we shall consider are one of the Wiman hyperelliptic curves $y^2 = x^{2 g+1} - 1$ and the Lefschetz curves of the form $x^m y^n + y^m + x^n = 0$ for coprime $m, n$ where $p := m^2 - mn + n^2 > 7$ is a prime congruent to 1 mod 3. 

    The former has automorphism group $C_{2 g+1} \times C_2$, which contains the normal subgroup $C_{2 g+1}$ of odd order which quotients to $\mathbb{P}^1$. 
    The latter has automorphism group of the form $C_p \rtimes C_3$, which contains the normal subgroup $C_p$ of odd order which quotients to $\mathbb{P}^1$, which is easiest seen by writing the curve in the from $\tilde{y}^p + \tilde{x}^a(1+\tilde{x})=0$ as can always be done for some $a$ \cite[p.~464]{Lefschetz1921a}.
\end{proof}

\begin{remark}
For the hyperelliptic family in the proof of Theorem \ref{thm: infinitely many UIC curves} we can say more in the case that $2 g+1$ is a prime $p$, as all characteristics that are not invariant are in orbits of order $p$. 

Using Riemann-Hurwitz we know that the $C_p$ subgroup acts with signature $(0; p^r)$ where 
\[
2\times\frac{p-1}{2} - 2 = p\psquare{-2 + r\pround{1 - \frac{1}{p}}} \Rightarrow r=3. 
\]
Using Proposition \ref{prop: parity of SOC UIC} we know the parity of the UIC is $3p(p - p^{-1})/8$, and hence is determined entirely by whether $p = \pm1 $ or $p = \pm 3$ mod 8.

\end{remark}
\section{Hurwitz Curves and the Method of Dolgachev}\label{sec: Hurwitz curves and Dolgachev}

\subsection{Machine Learning Predictions}

One use of the code for computing orbit decomposition united with the method of \cite{Behn2013} for computing rational representation from signatures is that we can investigate the application of machine learning techniques in illuminating the behaviour of orbits. Machine classification (namely a pipeline using a Standard Scaler followed by a Random Forest classifier\footnote{LDH is very grateful to Jacob Bradley for showing him how to do this.}) using features built from the group data and the signature alone achieved an accuracy of approximately 93\% in cross-validation when predicting if the corresponding group action gave a unique invariant characteristic when trained on data of over 1000 group actions, a far higher accuracy than the approximately $54\%$ one would expect if choosing randomly with prior knowledge of the frequency of actions with a UIC in the dataset. Specifically, we provided the features of genus, group order, whether the group action was large (in the sense $\abs{G} > 4(g-1)$), the maximum power of 2 dividing the group order, the number of involutions in the group, the number of involutions up to conjugacy, the number of entries in the signature, the number of even entries in the signature, the maximum entry of the signature, and the dimension of the corresponding family for 1326 group actions in genus $12$ or less. This suggests that from very basic heuristics alone one should be able to get strong results understanding the behaviour of UICs. Moreover, estimating feature importance using these methods showed that whether or not a group action was large for a given genus was unimportant in determining whether a given action led to a UIC. 

In order to lay a benchmark to guide future research we computed in Table \ref{tab: invariants prediction} the prediction of the classifier given the corresponding data for all the simple Hurwitz group with order $<10^6$, provided in \cite{Conder1987} (the $J_i$ are the first two Janko groups). These results are correct for the two Hurwitz curves known, though removing their data from the training set makes the classifier less accurate. Running a similar pipeline which instead predicts whether a group action has 0, 1, or many invariants characteristics confirms the predictions of Table \ref{tab: invariants prediction}, predicting those without a unique invariant characteristic have none. 

\begin{longtable}
{p{0.16\linewidth}|p{0.16\linewidth}|p{0.16\linewidth}}
\caption{Machine prediction of whether $I=1$, all simple Hurwitz groups of order $<10^6$} \\
    $G$ & $g$ & $I=1$ \\ \hline \hline 
    $\PSL_2(7)$ & 3 & True \\
    $\PSL_2( 8)$ & 7 & False \\
    $\PSL_2( 13)$ & 14 & True \\
    $\PSL_2( 27)$ & 118 & True \\
    $\PSL_2( 29)$ & 146 & True \\
    $\PSL_2( 41)$ & 411 & False \\
    $\PSL_2( 43)$ & 474 & True \\
    $J_1$ & 2091 & False \\
    $\PSL_2( 71)$ & 2131 & False \\
    $\PSL_2( 83)$ & 3404 & True \\
    $\PSL_2( 97)$ & 5433 & False \\
    $J_2$ & 7201 & False \\
    $\PSL_2( 113)$ & 8589 & False \\
    $\PSL_2( 125)$ & 11626 & True
\label{tab: invariants prediction}
\end{longtable}

In \S\ref{sec: method of Dol}-\ref{sec: sporadic examples} we will now show how to compute number of invariant characteristics for the curves in Table \ref{tab: invariants prediction} analytically using a new method, and in \S\ref{sec: comparison to prediction}, Table \ref{tab: invariants prediction answers}, we shall compare the answers. 

\subsection{The Method of Dolgachev}\label{sec: method of Dol}

Approximately four months after constructing these machine estimates we were alerted by Vanya Cheltsov to the paper \cite{Dolgachev1997a}, which allows one to get analytic solutions to the questions raised in Table \ref{tab: invariants prediction}, which we shall describe now. 

Denote with $\Pic(\mathcal{C})^G$ the group of $G$-invariant isomorphism classes of line bundles on $\mathcal{C}$, that is line bundles $L \to \mathcal{C}$ that admit a collection of isomorphisms $\phi_g : g^\ast L \to L$ for each $g \in G$. Moreover, let $\Pic(G; \mathcal{C})$ denote the group of $G$-linearised line bundles on $\mathcal{C}$ and corresponding linearisations, that is $G$-invariant line bundles that admit a collection of isomorphisms such that 
\[
\forall g, h \in G, \quad \phi_{g \circ h} = \phi_h \circ h^\ast \phi_g.
\]
Thinking in terms of divisors $\Div(\mathcal{C})$, principal divisors $\PDiv(\mathcal{C})$, and nonzero meromorphic functions $\mathcal{M}(\mathcal{C})^\times$ instead of line bundles, $\Pic(\mathcal{C})^G$ corresponds to $(\Div(\mathcal{C})/\PDiv(\mathcal{C}))^G$, while $\Pic(G; \mathcal{C})$ corresponds to $\Div(\mathcal{C})^G/ [\mathcal{M}(\mathcal{C})^\times]^G$ \cite[Corollary 2.3]{Dolgachev1997a}. The key results of Dolgachev are the following. 
\begin{prop}[\cite{Dolgachev1997a}, Proposition 2.2]
    We have a short exact sequence of abelian groups
    \begin{equation}\label{eq: Dol ES}
0 \to \Hom(G, \mathbb{C}^\times) \to \Pic(G; \mathcal{C}) \to \Pic(\mathcal{C})^G \to H^2(G, \mathbb{C}^\times) \to 0.
\end{equation}
Here $H^2(G, \mathbb{C}^\times)$ is the group cohomology viewing $\mathbb{C}^\times$ as a trivial $G$-module, which is known as the Schur multiplier of the group $G$. 
\end{prop}
\begin{proof}
	Dolgachev proves this using the spectral sequence of \cite[\S5.2]{Grothendieck1957}, but we shall give a more elementary approach here following \cite{Fried1989}. 
	
	Observe that the Short Exact Sequence (SES)
	\[
	0 \to \PDiv(\mathcal{C}) \to \Div(\mathcal{C}) \to \Pic(\mathcal{C}) \to 0, 
	\]
	gives rise to the Long Exact Sequence (LES)
	\[
	0 \to \PDiv(\mathcal{C})^G \to \Div(\mathcal{C})^G \to \Pic(\mathcal{C})^G \to H^1(G, \PDiv(\mathcal{C})) \to H^1(G, \Div(\mathcal{C})).
	\]
	We can use \cite[Lemma 4.1]{Fried1989} to say
	\begin{itemize}
		\item $H^1(G, \Div(\mathcal{C}))=0$, and 
		\item $H^1(G, \PDiv(\mathcal{C})) \cong H^2(G, \mathbb{C}^\times)$.
	\end{itemize}
	Moreover, from the SES
	\[
	0 \to \mathbb{C}^\times \to \mathcal{M}(\mathcal{C})^\times \to \PDiv(\mathcal{C}) \to 0,
	\]
	where $\mathcal{M}(\mathcal{C})$ are the meromorphic functions on the curve, we get the portion of the LES
	\begin{align*}
		0 &\to \mathbb{C}^\times \to [\mathcal{M}(\mathcal{C})^\times]^G \to \PDiv(\mathcal{C})^G \to \Hom(G, \mathbb{C}^\times) \to 0,
	\end{align*}
	where we have used \cite[Exercise 6.1.6]{Weibel1995} to equate $H^1(G, \mathbb{C}^\times) = \Hom(G, \mathbb{C}^\times)$. As such, quotienting by $[\mathcal{M}(\mathcal{C})^\times]^G$ we are done.
\end{proof}

\begin{remark}
    The map $\Pic(G; \mathcal{C}) \to \Pic(\mathcal{C})^G$ sends any $G$-linearised line bundle $L$ and its choice of linearisation to the isomorphism class of $L$. In this way, we think of $\Hom(G, \mathbb{C}^\times)$ as parametrising the possible $G$-linearisations of the trivial bundle $\mathcal{O}_\mathcal{C}$. 
\end{remark}

\begin{prop}[\cite{Dolgachev1997a}]\label{prop: structure theorem linearised pic}
    Suppose that $G$ acts on $\mathcal{C}$ with signature $(0; c_1, \dots, c_r)$. Then we have 
    \begin{equation}\label{eq: dolgachev SES}
    \Pic(G; \mathcal{C}) \cong \mathbb{Z} \oplus \psquare{\bigoplus_{i=1}^{r-1}\mathbb{Z}/(d_i/d_{i-1}) \mathbb{Z}},
    \end{equation}
    where 
\[
d_i = \gcd\pbrace{\prod_{j \in S} c_j \, | \, S \subset \pbrace{1, \dots, r}, \abs{S}=i}, \quad \text{taking} \quad d_0=1. 
\]
The generator of the $\mathbb{Z}$ factor $\gamma$ has underlying $G$-invariant line bundle (isomorphism class) $\Gamma$ determined by 
\[
K_\mathcal{C} = N\pround{r - 2 - \sum_{i=1}^r \frac{1}{c_i}} \Gamma , 
\]
where $N = \lcm(c_i)$.
\end{prop}
\begin{proof}
	We shall not recreate the entire proof but note that, as it shall be helpful later,  letting $Q_1, \dots, Q_r \in \mathbb{P}^1$ be the branch points of the projection $ \pi : \mathcal{C} \to \mathcal{C}/G$, then $\Pic(G; \mathcal{C})$ is generated by the $G$-invariant divisors 
	\[
	D_i = \pi^{-1}(Q_i) = \sum_{\substack{{P \in \mathcal{C}} \\ {\pi(P) = Q_i}}} P.
	\]
	These are subject to the equivalence relation $c_i D_i - c_j D_j = \pi^\ast(Q_i - Q_j) \sim 0$. Using Smith normal form to compute the abelian invariants yields the $d_i$. 
\end{proof}
\begin{corollary}\label{corr: 2-torsion kernel and cokernel}
	In the situation of Proposition \ref{prop: structure theorem linearised pic}, letting $e = \abs{\pbrace{\text{even $c_i$}}}$,  
	\begin{align*}
	\Pic(G; \mathcal{C}) / 2 \Pic(G; \mathcal{C}) &= \left \lbrace \begin{array}{cc}
	\mathbb{Z}_2, & e = 0, \\\mathbb{Z}_2^{e}, & e \neq 0.
	\end{array}\right . ,  \\
\Pic(G; \mathcal{C})_2 &= \left \lbrace \begin{array}{cc}
	0, & e = 0, \\ \mathbb{Z}_2^{e-1}, & e \neq 0.
\end{array}\right .  . 
	\end{align*}
\end{corollary}
\begin{proof}
	The first equality from writing 
	\[
	\Pic(G; \mathcal{C}) / 2 \Pic(G; \mathcal{C}) = \pangle{D_i, \, i=1, \dots, r \, | \, \forall i, j, \, c_i D_i - c_j D_j = 0 = 2 D_i}.
	\]
If all $c_i$ are odd, then we may use the second relation to reduce the first relation to $D_i = D_j$, and then the cokernel has presentation $\pangle{D_1 \, | \, 2 D_1 = 0}$. If there are $c_i$ which are even, say $c_1, \dots, c_e$ then they will drop out of the first relation to just give $D_i=0$ for $i>e$ and then the cokernel has presentation $\pangle{D_i, \, i=1, \dots, e \, | \, \forall i, \, 2D_i = 0}$. 
 
The second follows from observing that 
	\begin{equation}\label{eq: Torsion portion of linearised group}
	\mathbb{Z}/N \mathbb{Z} \oplus \psquare{\bigoplus_{i=1}^{r-1}\mathbb{Z}/(d_i/d_{i-1}) \mathbb{Z}}
	=\bigoplus_{i=1}^{r}\mathbb{Z}/(d_i/d_{i-1}) \mathbb{Z}
	=\bigoplus_{i=1}^{r}\mathbb{Z}/c_i \mathbb{Z}.
	\end{equation}
 The 2-torsion part of the left hand side is $\mathbb{Z}_{\gcd(2, N)} \oplus \Pic(G; \mathcal{C})_2$, while 2-torsion of the right hand side is $\mathbb{Z}_2^e$. The result follows using 
 \[
 \mathbb{Z}_{\gcd(2, N)} = \left \lbrace \begin{array}{cc}
     0, & e=0, \\
     \mathbb{Z}_2 & e \neq 0.
 \end{array} \right.
 \]
\end{proof}

As invariant theta characteristics would give elements of $\Pic(\mathcal{C})^G$ which square to $K_\mathcal{C}$, we can hope that in particular circumstances  Equation (\ref{eq: dolgachev SES}) may give us enough information  to determine the number of invariant characteristics. We have the following characterisations.
\begin{lemma}
	Suppose we are in the situation of Proposition \ref{prop: structure theorem linearised pic}. Then
\begin{itemize}
	\item an invariant characteristic exists if and only if the image of $K_\mathcal{C} \in \Pic(G;\mathcal{C})$ under the map $\Pic(G; \mathcal{C}) \to \Pic(\mathcal{C})^G$ is a square, and 
	\item if exactly $m>0$ invariant characteristic exists, then the 2-torsion subgroup $\Pic(\mathcal{C})^G_2$ has rank $m-1$.  
\end{itemize}
\end{lemma}

As computing the 2-torsion subgroup of $\Pic(\mathcal{C})^G$ will be important, we recall briefly a quick fact from homological algebra
\begin{lemma}\label{lemma: torsion snake}
Given a SES $0 \to A \to B \to C \to 0$ there is an associated LES 
\begin{equation}
0 \to A_2 \to B_2 \to C_2 \to A/2A \to B/2B \to C/2C \to 0
\end{equation}
where $A_2$ is the 2-torsion part of $A$ (and likewise for $B,C$).
\end{lemma}

\begin{example}[Cyclic Groups]
	Consider the case $G=C_n$ with quotient genus 0, which gives 
	$$
	\Hom(G,\mathbb{C}\sp\times)\cong C_n,\quad H^2(G,\mathbb{C}\sp\times)\cong 0,
	$$
	and 
	$$
	\Pic(G; \mathcal{C}) \cong \mathbb{Z} \oplus \psquare{\bigoplus_{i=1}^{r-1}\mathbb{Z}/(d_i/d_{i-1}) \mathbb{Z}},
	$$
	with \[
	K_{\mathcal{C}} = N\pround{r - 2 - \sum_i \frac{1}{c_i}} \Gamma. 
	\]
	Equation (\ref{eq: Dol ES}) thus becomes 
	\begin{equation}\label{sesks}
	0\to\mathbb{Z}_n \to \Pic(G; \mathcal{C}) \to \Pic(\mathcal{C})^G\to
	0.
	\end{equation}
	In the case of a cyclic group action which quotients to $\mathbb{P}^1$ \cite{Harvey1966} shows that $N = n$. Then Riemann-Hurwitz gives that $K_\mathcal{C}=2(g-1)\Gamma$ which shows that $(g-1)\Gamma$ is an invariant theta characteristic. We now need to count how many exist. 
	
    Applying Lemma \ref{lemma: torsion snake} to Equation (\ref{sesks}) we get 
	\begin{equation}\label{eq: DES cyclic case}
	0 \to \mathbb{Z}_{\gcd(2, n)} \to \Pic(G; \mathcal{C})_2 \to \Pic(\mathcal{C})^G_2 \to \mathbb{Z}_n/2\mathbb{Z}_n \to \Pic(G; \mathcal{C})/2\Pic(G; \mathcal{C}), 
	\end{equation}
	We analyse this in the case of $n$ odd and $n$ even. Let $e$ be the number of $c_i$ which are even. We will use Corollary \ref{corr: 2-torsion kernel and cokernel}.
	\begin{itemize}
		\item When $n$ is odd, $e=0$, and so Equation (\ref{eq: DES cyclic case}) becomes 
		\[
		0 \to 0 \to 0 \to \Pic(\mathcal{C})^G_2 \to 0 \to \mathbb{Z}_2,
		\]
		and immediately we see $\Pic(\mathcal{C})^G_2=0$, i.e. the invariant characteristic is unique. 
		\item When $n$ is even Equation (\ref{eq: DES cyclic case}) becomes 
		\[
		0 \to \mathbb{Z}_2 \to \mathbb{Z}_2^{e-1} \to \Pic(\mathcal{C})^G_2 \to \mathbb{Z}_2 \to \mathbb{Z}_2^{e}.
		\]
        At this stage we claim that the map $\Pic(\mathcal{C})^G_2 \to \mathbb{Z}_2$ is the zero map. To see this, suppose we have $[L] \in \Pic(\mathcal{C})^G_2$. Because $H^2(G, \mathbb{C}^\times)=0$ we know $[L] = [L^\prime]$ where $L^\prime$ is some $G$-linearisable line bundle. Pick a linearisation $\pbrace{\phi_g}$ of $L^\prime$, of which there are $n$ choices. We know $[L^2] =[L]^2 =  [\mathcal{O}_{\mathcal{C}}]$, so $L^2$ is the image of a linearisation $\pbrace{\psi_g}$ of the trivial line bundle (there are $n$ many of them), and it is this trivialisation the connecting map sends $[L]$ to. We have $\phi_g^2 \sim \psi_g$, and then $\pbrace{\psi_g}$ is necessarily trivial in the cokernel. Suppose we had instead chosen $\pbrace{\tilde{\phi}_g}$ as the linearisation of $L^\prime$. Then the two linearisations would differ by an element of $\Hom(G, \mathbb{C}^\times)$, and so $\phi_g^2$ differs from $\tilde{\phi}_g$ by an element of $2\Hom(G, \mathbb{C}^\times)$, so they are equal in the cokernel.
	\end{itemize}
\end{example}

\subsection{Hurwitz Curves}\label{sec: Hurwitz curves}

When $\mathcal{C}$ is a Hurwitz curve and $G$ is the full automorphism group, we know $\abs{G} = 84(g-1)$ and the signature is $(0; 2, 3, 7)$, which means that simply $\Pic(G; \mathcal{C}) = \mathbb{Z}\gamma$, and moreover that $\Gamma = K_\mathcal{C}$. Moreover, all Hurwitz groups are perfect (that is $\psquare{G, G} = G$) \cite{Conder1990}, and so $\Hom(G, \mathbb{C}^\times) = 0$. This reduces Equation (\ref{eq: dolgachev SES}) to 
\begin{equation}\label{eq: dolgachev SES simplified}
0 \to \mathbb{Z} K_\mathcal{C} \to \Pic(\mathcal{C})^G \to H^2(G, \mathbb{C}^\times) \to 0. 
\end{equation}
Obviously when $H^2(G, \mathbb{C}^\times)=0$, Equation (\ref{eq: dolgachev SES simplified}) is just 
\[
0 \to \mathbb{Z} K_{\mathcal{C}} \to \Pic(\mathcal{C})^G \to 0,
\]
and no invariant characteristics can exist, so we assume that the multiplier group is nontrivial from hereon in.

We can immediately apply Lemma \ref{lemma: torsion snake} to get
\[
0 \to \Pic(G; \mathcal{C})_2 \to \Pic(\mathcal{C})_2^G \to H^2(G, \mathbb{C}^\times)_2 \to \Pic(G; \mathcal{C})/ 2\Pic(G; \mathcal{C}),
\]
which can be rewritten as 
\begin{equation}\label{eq: torsion Hurwitz sequence reduced}
0 \to 0 \to \Pic(\mathcal{C})_2^G \to H^2(G, \mathbb{C}^\times)_2 \to \mathbb{Z}_2.
\end{equation}
Dolgachev provides the following helpful interpretation.
\begin{lemma}[\cite{Dolgachev1997}]\label{lemma: Dolgachev interpretation of connecting map}
Let $G$ be a perfect group acting with signature containing only a single even element $c_1=2$. The map $H^2(G, \mathbb{C}^\times)_2 \to \Pic(G; \mathcal{C})/ 2\Pic(G; \mathcal{C})$ is the restriction map $H^2(G, C_2) \to H^2(G_{P_1}, C_2)$ where $P_1 \in \mathcal{C}$ is a ramification point whose isotropy subgroup $G_{P_1} \cong C_2$.
\end{lemma}

When we restrict to looking at simple Hurwitz groups, we may use the fact that the Schur multipliers have been computed in the Atlas \cite{Conway1985} and are cyclic groups. In particular, for the finite simple groups which are Hurwitz groups \cite{Conder1990, Wilson1993, Wilson2001} we give the relevant multipliers in Table \ref{tab: finite simple Hurwitz groups}.
\begin{longtable}
{p{0.1\linewidth}|p{0.14\linewidth}|p{0.58\linewidth}}
\caption{All finite simple Hurwitz groups and their Schur multiplier} \\
    $G$ & $H^2(G, \mathbb{C}^\times)$ & Comments \\ \hline \hline 
    $\PSL_2(q)$ & $C_{\gcd(2, q-1)}$ & $q=7$, $q=p$ for prime $p=\pm 1 \mod 7$, or $q=p^3$ for prime $p = \pm 2, \pm 3 \mod 7$. \\
    $A_n$ & $C_2$ & $A_n$ Hurwitz for all but finitely many $n$ \\
    $\indices{^2}G_2(3^p)$ & trivial & $p$ a prime $> 3$\\
    $J_1$ & trivial & \\
    $J_2$ & $C_2$ & \\
    $Co_3$ & trivial & \\
    $He$ & trivial & \\
    $Ru$  & $C_2$ & \\
    $HN$ & trivial & \\
    $Ly$ & trivial & \\
    $Fi_{24}^\prime$ & $C_3$ & \\
    $Th$ & trivial & \\
    $Fi_{22}$ & $C_6$ & \\
    $J_4$ & trivial & \\
    $M$ & trivial & 
\label{tab: finite simple Hurwitz groups}
\end{longtable}
Our earlier results show that the Hurwitz curves with automorphism group $\PSL_2(8)$, $\indices{^2}G_2(3^p)$, $J_1$, $Co_3$, $He$, $HN$, $Ly$, $Th$, $J_4$, or $M$ have no invariant theta characteristics as the corresponding multiplier group is trivial. 

Now because $\Hom(G, \mathbb{C}^\times)=0$ we have that $\Pic(G; \mathcal{C})\cong \mathbb{Z}\gamma$ is a subgroup of $\Pic(\mathcal{C})^G$ of index $n := \abs{H^2(G, \mathbb{C}^\times)}$.  
Set $A=\Pic(G; \mathcal{C})\cong \mathbb{Z}$, 
  $B=\Pic(\mathcal{C})^G\cong \mathbb{Z}\oplus T$,
  where $T$ is the torsion subgroup,
  $C= H^2(G, \mathbb{C}^\times)\cong C_n$, and write
our exact sequence as
$$0\to \mathbb{Z} \overset{\phi}{\to} \mathbb{Z} \oplus T \overset{\psi}{\to} C_n \to 0,\qquad \phi(1)=(k,l).$$
Now if $0\rightarrow A\rightarrow B\rightarrow C$ is exact then so is
$0\rightarrow \Tor(A)\rightarrow \Tor(B)\rightarrow \Tor(C)$. Thus $T$ injects into the cyclic group  $C_n$ and so $T\cong C_c$ for some $c|n$. 
The matrix of generators and relations of the quotient $B/A=\Pic(\mathcal{C})^G/\Pic(G; \mathcal{C})$ is then $\begin{psmallmatrix}
 0 & c \\ k & l
\end{psmallmatrix}$. The Smith normal form of  this will be $C_{d_1}\times C_{d_2/d_1} $
where $d_1=\gcd(c,k,l)$, $d_2=ck$. For this quotient to be isomorphic to $C_n$, it certainly must be true that $ck=n$.

Now an invariant theta characteristic exists if and only if $k$ is even and there exists some $m$ such that $l = 2m \mod c$, which is to say that an invariant characteristic can never exist when $n$ is odd as then necessarily $k$ is odd. This tells us that Hurwitz curves with automorphism group $Fi^\prime_{24}$ have no invariant theta characteristics, and hereon we assume $n$ is even. Looking at Table \ref{tab: finite simple Hurwitz groups} we see that when $n$ is even, the maximum power of 2 dividing $n$ is just 2, and so we have 
\[
k \text{ even} \Leftrightarrow c \text{ odd} \Leftrightarrow \Pic(\mathcal{C})^G_2 = 0.
\]
As $c$ being odd is sufficient to be able to solve $l = 2m \mod c$ for some $m$, we see that we have an invariant characteristics, and moreover a unique one, when $\Pic(\mathcal{C})^G_2=0$. By using Lemma \ref{lemma: Dolgachev interpretation of connecting map} with Equation (\ref{eq: torsion Hurwitz sequence reduced}) this is equivalent to having the map $H^2(G, C_2) \to H^2(G_{P_1}, C_2)$ being injective, equivalently surjective.

\begin{prop}
Let $\mathcal{C}$ be a Hurwitz curve with simple automorphism group $G$. Then $\mathcal{C}$ has a unique invariant characteristic exactly when $H^2(G, \mathbb{C}^\times)\cong C_n$ has $n>0$ even and the map $H^2(G, C_2) \to H^2(G_{P_1}, C_2)$ is surjective. Otherwise, $\mathcal{C}$ has no invariant characteristics. 	
\end{prop}

One way to understand the map $H^2(G, C_2) \to H^2(G_{P_1}, C_2)$ is using the fact that isomorphism classes of central extensions of $G$ by an abelian group $A$, $0 \to A \overset{i}{\to} E \overset{\pi}{\to} G \to 0$, are in bijection with elements of $H^2(G, A)$ where $A$ is a trivial $G$-module. Given then inclusion $\imath :H \hookrightarrow G$, the restriction morphism $H^2(G, A) \to H^2(H, A)$ works via the diagram \cite[Exercise 6.6.4]{Weibel1995}
\begin{center}
\begin{tikzcd}
    0\arrow[r] & A\arrow[r]\arrow[d, "="] & E^\prime\arrow[r]\arrow[d] & H\arrow[r]\arrow[d, hook, "\imath"] & 0 \\ 0\arrow[r] & A\arrow[r] & E\arrow[r, "\pi"] & G\arrow[r] & 0,
\end{tikzcd}
\end{center}
sending the cocycle corresponding to the extension $E$ to the cocycle corresponding to the extension
\[
E^\prime := E \times_G H = \pbrace{(e, h) \in E \times H \, | \, \pi(e) = \imath(h)}. 
\]
There are two central extensions of $G = C_2$ by $A = C_2$, these are $E = V_4$ and $E = C_4$, with the latter corresponding to the generator of $H^2(C_2, C_2)$. 

Restricting to the case where $H^2(G, \mathbb{C}^\times)_2 = H^2(G, C_2) \cong C_2$, suppose $g = \gamma_1 \in G$ generates the subgroup $H=G_{P_1}$ and $\tilde{g}$ is one of the elements of $\pi^{-1}(g) \in E$ where $E$ is the unique nontrivial $C_2$ central extension of $G$, often denoted $2\cdot G$. We can define likewise $\tilde{e}$ to be the order-2 lift of the identity $e \in G$. We then have 
\[
E^\prime = \pbrace{(e, e), (\tilde{e}, e), (\tilde{g}, g), (\tilde{e}\tilde{g}, g)}.
\]
We define the \bam{lifting order} of $g$ to be the order of $\tilde{g}$ in $E$, and it depends only on the conjugacy class of $g$ \cite{Conway1985}. There are two possible cases:
\begin{enumerate}
    \item the lifting order is 2, so $\tilde{g}^2 = e$, and $E^\prime \cong V_4$, or 
    \item the lifting order is 4, to $\tilde{g}^2 = \tilde{e}$, and $E^\prime \cong C_4$.
\end{enumerate}
The latter is equivalent to the map $H^2(G, C_2) \to H^2(G_{P_1}, C_2)$ being surjective. As such, we have proven the following result.
\begin{prop}
    Given a Hurwitz curve $\mathcal{C}$ with simple automorphism group $G$, and $g \in G$ generating the stabiliser group $G_{P_1} \cong C_2$, either:
    \begin{enumerate}
        \item $H^2(G, \mathbb{C}^\times)$ is cyclic of odd order, and $\mathcal{C}$ has no invariant characteristics,
        \item $H^2(G, \mathbb{C}^\times)$ is cyclic of even order, the lifting order of $g$ in $2 \cdot G$ is 2, and $\mathcal{C}$ has no invariant characteristics, or
        \item $H^2(G, \mathbb{C}^\times)$ is cyclic of even order, the lifting order of $g$ in $2 \cdot G$ is 4, and $\mathcal{C}$ has exactly one invariant characteristic.
    \end{enumerate}
\end{prop}

We will now consider distinguish between the latter two cases for the curves in Table \ref{tab: finite simple Hurwitz groups}. 
\subsection{\secmath{\PSL_2(q)}}

First consider the case of $G = \PSL_2(q)$. The only even value of $q$ which gives a Hurwitz group is $q=8$, corresponding to the Fricke-Macbeath curve. In this case the Schur multiplier group is trivial, and so we have analytically shown that the Fricke-Macbeath curve has 0 invariant characteristics. 

Restrict then to the case when $q$ is odd. The generator of the group $$H^2(\PSL_2(q), C_2) \cong C_2$$ corresponds to the extension 
\[
1 \to C_2 \to \SL_2(q) \overset{\pi}{\to} \PSL_2(q) \to 1. 
\]
\begin{prop}
    The lifting order of any involution in $\PSL_2(q)$ to $\SL_2(q)$ is 4. 
\end{prop}
\begin{proof}
    One can show easily by exhaustively writing out entries that involutions in $\SL_2(q)$ are diagonal of the form $\diag(a, a^{-1})$ where $a^2=1$. There are $n$th roots of unity in $\operatorname{GF}(q)$ if and only if $n|q-1$, and in that event there are $\gcd(n, q-1)$ many roots. We are focusing on square roots of units, which will exist when $q$ is odd as we have, and then the $\gcd(2, q-1)=2$ square roots must be $\pm 1$. As such $\tilde{g}^2 = 1 \Rightarrow \tilde{g} = \pm 1 \Rightarrow g$ is the identity in $\PSL_2(q)$. Hence any involution must have lifting order 4.
\end{proof}
\begin{corollary}
    Hurwitz curves with automorphism group $\PSL_2(q)$, $q$ odd, have a unique invariant characteristic. 
\end{corollary} 

\subsection{\secmath{A_n}}\label{sec: An invariant characteristics}

The alternating group $A_n$ is a Hurwitz group for all $n$ except 64 values in the range $1 \leq n \leq 167$ \cite{Conder1980}. The group $S_n$ has the $n$-dimensional permutation representation which splits into a sum of a 1-dimensional acting on vectors with equal entries and a $(n-1)$-dimensional representation acting on vectors whose entries sum to 0 (that is on an $n-1$ simplex), and so this gives an embedding $A_n \hookrightarrow \SO(n-1)$. Restricting to $n>4$, the nontrivial double cover of $\SO(n-1)$ is the well known group $\operatorname{Spin}(n-1)$, and so we get a nontrivial double cover of $A_n$ denoted $2\cdot A_n$ as 
\begin{center}
\begin{tikzcd}
    0\arrow[r] & C_2 \arrow[r]\arrow[d, "="] & 2 \cdot A_n\arrow[r]\arrow[d] & A_n \arrow[r]\arrow[d, hook, "\imath"] & 0 \\ 0\arrow[r] & C_2 \arrow[r] & \operatorname{Spin}(n-1) \arrow[r, "\pi"] & \SO(n-1) \arrow[r] & 0.
\end{tikzcd}
\end{center}
This is the nontrivial element of $H^2(A_n, C_2) \cong C_2$.

We can then reduce the question to asking about whether, given nontrivial $g \in \SO(n-1)$ such that $g^2 = 1$, under what conditions can we have $\tilde{g}^2 = 1$ where $\tilde{g} \in \operatorname{Spin}(n-1)$ is a lift of $g$. 

\begin{lemma}
    $\tilde{g}^2=1$ if and and only if $g$ is a product of $4k$ disjoint transpositions for some integer $k$. 
\end{lemma}
\begin{proof}
This is proven in \cite[Proposition 5.10]{Bailey2002}, though we shall take a different approach here. Using the geometric interpretation of multiplication in $\operatorname{Spin}(n-1)$ as coming from lifts of paths in $\SO(n-1)$, we see that if we take a path $\gamma : [0,1] \to \SO(n-1)$ such that $\gamma(0)=1$, $\gamma(1)=g$, then $\gamma^2 : [0,1] \to \SO(n-1)$ is a path with $\gamma^2(0) = 1 = g^2 = \gamma^2(1)$, such that when we lift this to a path $\widetilde{\gamma^2} : [0,1] \to \operatorname{Spin}(n-1)$ with $\widetilde{\gamma^2}(0)=1$, $\tilde{g}^2 := \widetilde{\gamma^2}(1) \in \pi^{-1}(1)$. This tells us that $\tilde{g}^2 = 1$ precisely when the path $\gamma^2$ is contractible. Now we may diagonalise $g$ as 
\[
g = \diag(\underbrace{-1}_{\times 2l}, \underbrace{1}_{\times m}),
\]
where we must have $2l+m = n-1$ and $m-2l = \Tr(g)$, which together determine $4l=n-1-\Tr(g)$ and $2m=n-1+\Tr(g)$. We can work out $\Tr(g)$ using some representation theory when $g$ corresponds to an element of $A_n$. Recall that in the permutation representation $\chi(g) = \abs{\operatorname{Fix}(g)}$ is the number of fixed points of $g$. Because the permutation representation decomposes as $ \underline{n-1} \oplus \underline{1}$ we must have $\Tr(g) = \abs{\operatorname{Fix}(g)}-1$. 
Thus $\abs{\operatorname{Fix}(g)} = n - 4l$.
As $g$ is an involution, its cycle type will consist of just transpositions, and so 
$\abs{\operatorname{Fix}(g)} = n - 2(\#\text{transpositions})$; indeed, as $g \in A_n$ it must be a product of an even number of transpositions.
We have therefore that $g$ consists of $2l$
transpositions.

Now each $\diag(-1, -1)$ block in $g$ corresponds to a rotation of a disjoint plane by $\pi$ radians, and the path in $\SO(n-1)$ corresponding to the rotation about this axis by an angle increasing from $0$ to $2\pi$ is non-contractible. The fundamental group of $\SO(n-1)$ is $\pi_1(\SO(n-1)) = \mathbb{Z}_2$, so a composition of non-contractible loops will be contractible if there are an even number of them. Hence the overall path $\gamma^2$ will be contractible if and only if $l$ is even. 
\end{proof}
\begin{corollary}
    The lifting order of $g$ is 4, and hence the corresponding $A_n$ Hurwitz curve has a UIC, if and only if $l$ is odd, when $g$ is written as the product of $2l$ disjoint $2$-cycles.
\end{corollary}

We may give a relatively simple criterion to identify the two cases. In particular, when $A_n$ is a Hurwitz group let $\gamma_i$, $i=1,2,3$, be the generators of order $2,3,7$ respectively in the corresponding generating vector, and denote $f_i = \abs{\operatorname{Fix}(\gamma_i)}$. Then \cite{Conder1980, Conder1984} there exists $h \in \mathbb{Z}_{\geq 0}$ such that 
\[
n = 84(h-1) + 21f_1 + 28f_2 + 36f_3. 
\]
Using that $n-f_1 = 4l$, we get a simple lemma.
\begin{lemma} 
The parity of $l$ is determined by 
\[
2 | l \Leftrightarrow h-1 + f_1 + f_2 + f_3 \equiv 0 \mod 2. 
\]
\end{lemma}
The five smallest values of $n$ for which we get that $A_n$ is a Hurwitz group are found in \cite{EtayoGordejuela2005} and yield Table \ref{tab: UIC for An} giving the number of invariant characteristics $I$.

\begin{longtable}{c|c}\caption{First five $A_n$ Hurwitz groups and the number of invariant characteristics} \\
 $(n; h, f_1, f_2, f_3)$ & $I = h-1+f_1+f_2 + f_3 \mod 2$ \\ \hline \hline
 $(15; 0, 3, 0, 1)$ & 1 \\
 $(21; 0, 1, 3, 0)$ & 1 \\
 $(22; 0, 2, 1, 1)$ & 1 \\
 $(28; 0, 4, 1, 0)$ & 0 \\
 $(29; 0, 1, 2, 1)$ & 1\label{tab: UIC for An}
 \end{longtable}

This shows that the question of whether an $A_n$ Hurwitz curve has a UIC depends on $n$ in a way that is not immediately clear. More research is required to determine the structure of the involution which generates $A_n$ as a Hurwitz group. 

\subsection{Sporadic Examples}\label{sec: sporadic examples}

It remains to determine the number of invariant characteristics for the sporadic simple groups 
$J_2$, $Ru$ and $Fi_{22}$. The Atlas \cite{Conway1985} records the lifting orders of conjugacy classes of $G$ to the double cover $2 \cdot G$, and so we need only to identify the conjugacy class of the generator of $G_{P_1}$. 
\begin{itemize}
    \item $J_2$ has two conjugacy classes of involutions, $2A$ and $2B$, with lifting order 2 and 4 respectively in $2\cdot J_2$. It is known that the $2B$ conjugacy class is that which generated $J_2$ as a Hurwitz group \cite{Finkelstein1973, Broughton1991}, and so a $J_2$ Hurwitz curve has a UIC. 
    \item $Ru$ has two conjugacy classes of involutions, $2A$ and $2B$, which have lifting order 2 and 4 respectively in $2\cdot Ru$. It is shown in \cite{Darafsheh2003} that $Ru$ is $(2B, 3, 7)$-generated, but they do not answer the question of whether it is $(2A, 3, 7)$-generated. Using the ``Brauer trick" argument of \cite{Broughton1987} with respect to the character $\chi_4$ of \cite{Conway1985} shows that in fact $Ru$ is not $(2A, 3, 7)$-generated, and so any Hurwitz curve with automorphism group $Ru$ has a UIC. 
    \item $Fi_{22}$ has three conjugacy classes of involutions, but all of them have lifting order $2$ in $2 \cdot Fi_{22}$, and so any Hurwitz curve with automorphism group $Fi_{22}$ has no invariant theta characteristics. 
\end{itemize}

\begin{remark}
    Had the lifting orders not been computed in the Atlas, we would have been able to work them out by taking a presentation of the group in an appropriate alternating group, and then applying the method of \S\ref{sec: An invariant characteristics}. 
\end{remark}

\subsection{Comparison to Machine Prediction}\label{sec: comparison to prediction}

We can now make a copy of Table \ref{tab: invariants prediction} which gives also the true value of $I$. We see that the machine estimation had an accuracy of $9/14 \approx 64\%$, with zero false positives. It is not surprising that the machine classification performed poorly on the subset of simple Hurwitz groups, as these are acting in a range of genera and signatures poorly represented in the training dataset. 

\begin{longtable}
{p{0.16\linewidth}|p{0.16\linewidth}|p{0.16\linewidth}|p{0.16\linewidth}}
\caption{Machine prediction of whether $I=1$, all simple Hurwitz groups of order $<10^6$, and true value of $I$} \\
    $G$ & $g$ & $I=1$ prediction & $I$ \\ \hline \hline 
    $\PSL_2(7)$ & 3 & True & 1\\
    $\PSL_2( 8)$ & 7 & False & 0\\
    $\PSL_2( 13)$ & 14 & True & 1\\
    $\PSL_2( 27)$ & 118 & True & 1\\
    $\PSL_2( 29)$ & 146 & True & 1\\
    $\PSL_2( 41)$ & 411 & False & 1\\
    $\PSL_2( 43)$ & 474 & True & 1\\
    $J_1$ & 2091 & False & 0\\
    $\PSL_2( 71)$ & 2131 & False & 1\\
    $\PSL_2( 83)$ & 3404 & True & 1\\
    $\PSL_2( 97)$ & 5433 & False & 1\\
    $J_2$ & 7201 & False & 1\\
    $\PSL_2( 113)$ & 8589 & False & 1\\
    $\PSL_2( 125)$ & 11626 & True & 1
\label{tab: invariants prediction answers}
\end{longtable}

A surprising fact that we have encountered so far is that the number of invariant characteristics on a Hurwitz curve just depends on the group and not its generating vector, i.e. the topological type of its action. It is known (see for example \cite{Conder1987}) that the Hurwitz groups can have topologically distinct generating vectors.

\section{Outlook}

In this work we've laid the groundwork for a more developed theory of invariant characteristics, moving beyond the work of \cite{Atiyah1971, Kallel2010, Beauville2013} which considered only those invariant under a single automorphism. We have approached this from two separate directions: by rephrasing the question of invariant characteristics in terms of group cohomology of an affine action; and by understanding the whole group of invariant line bundles. At present, it is not clear how the two may be linked, and clarifying this will undoubtedly serve to help the development of theory. 

At present, sufficient conditions for unique invariant characteristics have been found, but not any strong necessary conditions. In our current use of the group cohomology method, we have not utilised much of the theory of the rational representation, and it seems likely that better understanding this will be helpful. This is especially pressing as we provide examples through our computation of curves for which the SOC property does not explain the existence of a UIC, though the property is enough to cover the majority of curves seen.

With regards to Hurwitz curves with a simple automorphism group, though we have provided a criterion for checking whether certain generating vectors of $A_n$ Hurwitz curves yield a UIC, it is not clear whether a general statement can be made. Further research should go into investigating what generating vectors of $A_n$ as a Hurwitz group can occur. 

For future work, one could also generalise to consider \bam{$r$-spin structures}, which are $r$th-roots of the canonical bundle. These exist when $r | 2g-2$, and in principle our methods can be extended to study the orbits of these.

\appendix

\section{Tables of Orbits}\label{sec: tables of orbits}

In this appendix we provide all the tables of orbit decompositions described in \S\ref{sec: computations of orbit tables}. We shall use the notation of \cite{Bars2012, Badr2016} that $L_d(x, y, \dots)$ is a generic homogeneous degree-$d$ polynomial in the arguments. Moreover, we shall make some comment about the completeness of the data.
\begin{itemize}
    \item The complete list of genus-2 curves with nontrivial reduced automorphism group comes from \cite{Bolza1887}.
    \item The complete list of non-hyperelliptic genus-3 curves with nontrivial (reduced) automorphism group comes from \cite{Bars2012}. Bars attributes the first work completing this to Henn in 1976, but Wiman appears to have completed the calculation earlier in \cite{Wiman1895a}. Bars and Dolgachev disagree on the automorphism group of the curve given by $f = 1 + y^4 + x^4 + a(x^2 + y^2) + bx^2 y^2$; using Sage we find agreement with Bars. 
    \item The list of hyperelliptic curves with many automorphisms comes from \cite{Shaska2007, Muller2021}. The latter reference will also be used for higher genera. 
    \item \cite[Table 4]{Shaska2007} was used to verify the signatures of the non-hyperelliptic actions.
    \item The complete list of non-hyperelliptic genus-4 curves with nontrivial (reduced) automorphism group comes from \cite{Wiman1895}. Wiman distinguishes his curves by whether they lie on the nonsingular quadric or the cone, and we have followed this putting those that lie on the cone in the first portion of the table. For the curves which lie on the nonsingular quadric Wiman described the curve by providing the quadric and cubic, hence one must use the resultant to get a single plane equation, and this may require a projective linear transformation to find a nondegenerate coordinate system. We find a typo in Wiman's curve (8) with octahedral symmetry.
    \item Above genus 4 we are unaware of any complete lists classifying curves by their automorphism groups and giving plane models of the curves. Curves of genus $(d-1)(d-2)/2$ for $d \geq 3$ stand out because of the Pl\"ucker formula. In order to implement the numerical method for computing the rational representation it is also necessary to have a model of the curve with coefficients in $\mathbb{Q}$, and so this further limits the possible curves we may investigate. 
    \item As the LMFDB does not contain the data of signatures with quotient genus $>0$ at genera $>4$, we shall sometimes omit the signature of the action where it is unknown. 
    \item The first non-hyperelliptic curve of genus 5 used in Table \ref{tab: orbit decomposition, genus 5} is the family of Humbert curves given in \cite[(5.9)]{Kani1989}, the second from \cite{Swinarski2016}, and the third comes from taking resultants of the polynomials provided in \cite{Wiman1895}, hence one may choose to call it the Wiman octic for want of a better name. 
    \item The non-hyperelliptic genus-6 curves come from a list of all nonsingular plane quintics in \cite{Badr2016}, with the exception of the curve with automorphism group $(48, 15)$ which is from \cite{Swinarski2016}. The curve with automorphism group of order 150 is the Fermat quintic curve which has maximal automorphism group for a genus-6 curve \cite{Magaard2002}, the curve with automorphism group $S_5$ is the Wiman sextic \cite{Wiman1895, Edge1981a}, and the curve with automorphism group of order 72 is given in the LMFDB with label 6.72-15.0.2-4-9.1. The curve with automorphism group of order 39 is attributed to Snyder in \cite[p.~464]{Lefschetz1921a}, where it is constructed in a manner similar to Klein's curve. 
    \item The curve with automorphism group $\PSL_2(\mathbb{F}_8)$ is the Fricke-Macbeath curve \cite{Fricke1899, Macbeath1965}, the unique Hurwitz curve of genus 7, the rational plane model of which is a attributed to Brock in \cite{Hidalgo2017}. The remaining curves come from \cite[Table 2, Table 5]{Zomorrodian2010}. Table 5 in Zomorrodian gives all possible automorphism groups of non-hyperelliptic genus-7 curves where $\abs{\Aut} < 65$, and a plane curve form for each; this list contains some errors, for example a typo in curve 4 and the fact that curve 8 is hyperelliptic (as checked with Maple \cite{Maple2022}).
    \item We are unaware of any plane models of non-hyperelliptic curves of genus 8 and so examples in this genus are sadly missing; one can in principle get such models from the methods of \cite{Shimura1995}, using Sage's modular symbol functionality, but in practice the process of going from a canonical embedding to a plane model becomes infeasible. Likewise, one could use the methods of \cite{Swinarski2016} to get the canonical embedding, but the problem of finding a plane form from this remains. 
    \item The genus-9 curve with $\Autb=S_5$ is the Fricke octavic curve, defined in \cite{Edge1984}, constructed similarly to Bring's curve in $\mathbb{P}^3$ and so a plane form of the curve is found using resultants and a judicious choice of projective transformation to find a well conditioned coordinate system.
    \item The genus-9 curve with automorphism group of order 57 is a generalisation of Klein's curve and Snyder's curve \cite[p.~464]{Lefschetz1921a}.
\end{itemize}


\begin{longtable}
{p{0.37\linewidth}|p{0.16\linewidth}|p{0.16\linewidth}|p{0.16\linewidth}|p{0.04\linewidth}}
\caption{Orbit decomposition, all non-hyperelliptic and hyperelliptic genus-3 curves} \\
    $f$ & $\Autb$, $\bm{c}$ & Odd & Even & I\\ \hline \hline 
    $1 + L_2(x,y) + L_4(x,y)$ & $C_2$, $(1; 2^4)$ & $1_4$, $2_{12}$ & $1_{12}$, $2_{12}$ & 16 \\
    $L_1(x,y) + L_3(x,y)$ & $C_3$, $(0; 3^5)$ & $1_1$, $3_9$ & $3_{12}$ & 1 \\
    $1 + y^4 + x^4 + (ay^2 + bx^2) + cx^2 y^2$ & $V_4$, $(0; 2^6)$ & $2_6, 4_4$ & $1_8, 2_6, 4_4$ & 8 \\
    $bx^2 y^2 + x^3 + y^3 + axy + 1$ & $S_3$, $(0; 2^4, 3)$ & $1_1, 3_3, 6_3$ & $1_3, 3_9, 6_1$ & 4 \\
    $y^4 + x^3 + ay^2 + 1$ & $C_6$, $(0; 2, 3^2, 6)$ & $1_1, 3_1, 6_4$ & $3_4, 6_4$ & 1 \\
    $1 + y^4 + x^4 + a(x^2 + y^2) + bx^2y^2$ & $D_4$, $(0; 2^5)$ & $4_5, 8_1$ & $1_4, 2_4, 4_4, 8_1$ & 4 \\
    $xy^3 + x^3 + 1$ & $C_9$, $(0; 3, 9^2)$ & $1_1, 9_3$ & $9_4$ & 1 \\
    $y^4 + ay^2 + x^4 + 1$ & $( C_4~\times~C_2 ) \rtimes C_2 {\cong (16, 13)}$, $(0; 2^3, 4)$ & $4_1, 8_3$ & $2_6, 8_3$ & 0 \\
    $1 + y^4 + x^4 + a(y^2 + x^2 + y^2x^2)$ & $S_4$, $(0; 2^3, 3)$ & $4_1, 12_{2}$ & $1_{2}$, $3_{2}$, $4_1$, $6_2$, $12_1$ & 2 \\
    $x^4 + y^4 + x$ & $((C_4 \times C_2) \rtimes {C_2) \rtimes C_3}$ ${\cong (48, 33)}$, $(0; 2, 3, 12)$ & $4_1, 24_1$ & $6_2, 24_1$ & 0 \\
    $y^4 + x^4 + 1$ & $(C_4^2~\rtimes~C_3) \rtimes C_2 {\cong (96, 64)}$, $(0; 2, 3, 8)$ & $12_1, 16_1$ & $4_2, 12_1, 16_1$ & 0 \\
    $xy^3 + x^3 + y$ & $\PSL_3(\mathbb{F}_2)$, $(0; 2, 3, 7)$ & $28_1$ & $1_1, 7_2, 21_1$ & 1 \\ \hline
    $y^{2}- (x^{8} +a x^{6} +b x^{4} +c x^{2} + 1)$ & $C_2$, $(1; 2^4)$ & $1_{4}$, $2_{12}$ & $1_{12}$, $2_{12}$ & $16$ \\
    $y^2 -x(x^2-1)(x^4 + ax^2 + b)$ & $C_2$ & $1_{4}$, $2_{12}$ & $1_{4}$, $2_{16}$ & $8$ \\
    $y^2 - (x^4 + ax^2 + 1)(x^4 + bx^2 + 1)$ & $V_4$, $(0; 2^6)$ & $2_{6}$, $4_{4}$ & $1_{8}$, $2_{6}$, $4_{4}$ & $8$ \\
    $y^2 - (x^4 - 1)(x^4 + ax^2 + 1)$ & $V_4$ & $1_{2}$, $2_{3}$, $4_{5}$ & $1_{2}$, $2_{7}$, $4_{5}$ & $4$ \\
    $y^2 - x(x^6 + ax^3 + 1)$ & $S_3$, $(0; 2^4, 3)$ & $1_{1}$, $3_{3}$, $6_{3}$ & $1_{3}$, $3_{9}$, $6_{1}$ & $4$ \\
    $y^2 - (x^8 + ax^4 + 1)$ & $D_4$, $(0; 2^2, 4^2)$ & $4_{5}$, $8_{1}$ & $1_{4}$, $2_{4}$, $4_{4}$, $8_{1}$ & $4$ \\
    $y^{2} - (x^7 - 1)$ & $C_7$, $(0; 7^3)$ & $7_{4}$ & $1_{1}$, $7_{5}$ & $1$ \\
    $y^{2} - x(x^{6} - 1)$ & $D_6$ & $1_{1}$, $3_{1}$, $6_{2}$, $12_{1}$ & $1_{1}$, $2_{1}$, $3_{1}$, $6_{5}$ & $2$ \\
    $y^{2} - (x^{8} - 1)$ & $D_8$ & $4_{1}$, $8_{3}$ & $1_{2}$, $2_{1}$, $4_{2}$, $8_{3}$ & $2$ \\
    $ y^2 - (x^8 + 14x^4 + 1)$ & $S_4$, $(0; 3, 4^2)$ & $4_1, 12_{2}$ & $1_{2}$, $3_{2}$, $4_1$, $6_2$, $12_1$ & 2 
\label{tab: orbit decomposition, genus 3}
\end{longtable}

\begin{longtable}
{p{0.37\linewidth}|p{0.16\linewidth}|p{0.16\linewidth}|p{0.16\linewidth}|p{0.04\linewidth}}
\caption{Orbit decomposition, all non-hyperelliptic genus-4 curves stratified by whether the corresponding quadric is singular, and separately some hyperelliptic genus-4 curve with many automorphisms} \\
$f$ & $\Autb$ & Odd & Even & I\\ \hline \hline 
$y^3 + y(ax^4 + bx^2 + c) + (dx^6 + ex^4 + fx^2 + g)$ & $C_2$, $(1; 2^6)$ & $1_{24}$, $2_{48}$ & $1_{40}$, $2_{48}$ & $64$ \\
$y^3 + y(ax^{4} + b x^{2} + c )+ d x(x^{4} + e x^{2} + f) $ & $C_2$, $(2; 2^2)$ & $2_{60}$ & $1_{16}$, $2_{60}$ & $16$ \\
$y^3 + y[a(x^{4} + 1) + b x^{2} ] + x[c(x^{4} + 1) + d x^{2}] $ & $V_4$, $(1; 2^3)$ & $2_{24}$, $4_{18}$ & $1_{16}$, $2_{24}$, $4_{18}$ & $16$ \\
$y^3 + y[a(x^{4} + 1) + b x^{2}] + x(x^{4} - 1)$ & $V_4$, $(1; 2^3)$ & $4_{30}$ & $1_{4}$, $2_{18}$, $4_{24}$ & $4$ \\
$y^3 + a yx^2 + x(x^{4} + 1) $ & $D_4$, $(0; 2^4, 4)$ & $4_{12}$, $8_{9}$ & $1_{4}$, $2_{6}$, $4_{18}$, $8_{6}$ & $4$ \\
$y^3 + y(x^4 + a) + (b x^{4} + c) $ & $C_4$, $(0; 2, 4^4)$ & $1_{4}$, $2_{10}$, $4_{24}$ & $1_{4}$, $2_{18}$, $4_{24}$ & $8$ \\
$y^3 + ayx^{2} + (x^{6} + b x^{3} + 1)$ & $S_3$, $(0; 2^6)$ & $1_{6}$, $3_{18}$, $6_{10}$ & $1_{10}$, $3_{30}$, $6_{6}$ & $16$ \\
$y^3 + a yx^2+ (x^{6} + 1)$ & $D_6$, $(0; 2^5)$ & $2_{3}$, $6_{13}$, $12_{3}$ & $1_{4}$, $2_{3}$, $3_{12}$, $6_{9}$, $12_{3}$ & $4$ \\
$y^3 + y( a x^{3} + b) + (x^{6} + c x^{3} + d) $ & $C_3$, $(1; 3^3)$ & $1_{3}$, $3_{39}$ & $1_{1}$, $3_{45}$ & $4$ \\
$y^3 + a y (x^{3} + 1) + (x^{6} + 20 x^{3} - 8)$ & $A_4$, $(0; 2, 3^3)$ & $4_{3}$, $12_{9}$ & $1_{1}$, $3_{1}$, $6_{6}$, $12_{8}$ & $1$ \\
$y^3 + ay + (x^{6} + b )$ & $C_6$, $(0; 2, 6^3)$ & $1_{3}$, $3_{7}$, $6_{16}$ & $1_{1}$, $3_{13}$, $6_{16}$ & $4$ \\
$y^{3} + y + x^{6}$ & $C_{12}$, $(0; 4, 6, 12)$ & $1_{1}$, $2_{1}$, $3_{1}$, $6_{3}$, $12_{8}$ & $1_{1}$, $3_{1}$, $6_{6}$, $12_{8}$ & $2$ \\
$y^3 + a y + (x^{5} + b) $ & $C_5$, $(0; 5^4)$ & $5_{24}$ & $1_{1}$, $5_{27}$ & $1$ \\
$y^3 + y +x^{5}$ & $C_{10}$, $(0; 5, 10^2)$ & $10_{12}$ & $1_{1}$, $5_{3}$, $10_{12}$ & $1$ \\
$y^3 - (x^{6} + ax^{5} +b x^{4} +c x^{3} +d x^{2} +e x +f ) $ & $C_3$, $(0; 3^6)$ & $3_{40}$ & $1_{1}$, $3_{45}$ & $1$ \\
$y^3 -( x^{6} + ax^{4} + b x^{2} + 1)$ & $C_6$, $(0; 2^2, 3^3)$ & $3_{8}$, $6_{16}$ & $1_{3}$, $3_{13}$, $6_{16}$ & $1$ \\
$y^3 -x(x^{4} +a x^{2} + 1)$ & ${C_6 \times C_2}$, $(0; 2^2, 3, 6)$ & $6_{8}$, $12_{6}$ & $1_{1}$, $3_{5}$, $6_{8}$, $12_{6}$ & $1$ \\
$y^3 -(x^{6} + ax^{3} + 1)$ & ${C_3 \times S_3}$, $(0; 2^2, 3^2)$ & $3_{2}$, $6_{1}$, $9_{6}$, $18_{3}$ & $1_{1}$, $3_{3}$, $9_{10}$, $18_{2}$ & $1$ \\
$y^3-(x^{5} + 1)$ & $C_{15}$, $(0; 3, 5, 15)$ & $15_{8}$ & $1_{1}$, $15_{9}$ & $1$ \\
$y^3-(x^{6} + 1)$ & ${C_6 \times S_3}$, $(0; 2, 6^2)$ & $6_{2}$, $18_{4}$, $36_{1}$ & $1_{1}$, $3_{1}$, $6_{1}$, $9_{4}$, $18_{3}$, $36_{1}$ & $1$ \\
$y^3-x(x^{4} + 1)$ & ${C_3 \times S_4}$, $(0; 2, 3, 12)$ & $12_{2}$, $24_{1}$, $36_{2}$ & $1_{1}$, $3_{1}$, $12_{2}$, $18_{2}$, $36_{2}$ & $1$ \\ \hline
$y^4(x+1) +y^3 (x^2 + ax + 1)+ y^2[b(x^{3} +1) + cx(x+1)] + y[d x(x^{2} + 1) + ex^2] + f x^{2}(x+1) $ & $C_2$, $(1; 2^6)$ & $1_{24}$, $2_{48}$ & $1_{40}$, $2_{48}$ & $64$ \\
$y^6 + y^{4}(x^2 + ax + 1) + y^2x(d x^{2} + b x +e) + c x^{3} $ & $C_2$, $(2; 2^2)$ & $2_{60}$ & $1_{16}$, $2_{60}$ & $16$ \\
$y^6 + y^{4}(x^2 + ax + 1) + y^{2}x[d(x^2 + 1) + b x] + c x^{3}$ & $V_4$, $(1; 2^3)$ & $2_{24}$, $4_{18}$ & $1_{16}$, $2_{24}$, $4_{18}$ & $16$ \\
$y^{6} + y^2[c x (xy^2 + 1) + b (x^{3} + y^{2}) + ax( y^{2} + x)] + x^{3} $ & $V_4$, $(1; 2^3)$ & $4_{30}$ & $1_{4}$, $2_{18}$, $4_{24}$ & $4$ \\
$b y^2(y^2 - x) (x^{2}-1) - y^{6} - axy^2( y^{2} + x^{2}) - x^{3} $ & $D_4$, $(0; 2^4, 4)$ & $4_{12}$, $8_{9}$ & $1_{4}$, $2_{6}$, $4_{18}$, $8_{6}$ & $4$ \\
$y^6 + ay^{3} (x^{3}+1) + b x y^{4} + c x^{2} y^{2} + x^{3}$ & $S_3$, $(0; 2^6)$ & $1_{6}$, $3_{18}$, $6_{10}$ & $1_{10}$, $3_{30}$, $6_{6}$ & $16$ \\
$y^6 + a y^{3}(x^3 + 1) + b x y^{4} + b x^{2} y^{2} + x^{3}$ & $D_6$, $(0; 2^5)$ & $2_{3}$, $6_{13}$, $12_{3}$ & $1_{4}$, $2_{3}$, $3_{12}$, $6_{9}$, $12_{3}$ & $4$ \\
$by^2(y^2-x)(x^2-1) - (y^2+x)^3 $ & $S_4$, $(0; 2^3, 4)$ & $4_6$, $12_6$, $24_1$ & $1_4$, $4_6$, $6_6$, $12_6$ & 4 \\
$x^{3} y^{3} + y^{6} + (a+b+1) y^2(x^{3} - y^{3}) + (ab+a+b) y(x^{3} - y^{3}) + ab( x^{3} - y^{3})$ & $S_3$, $(0; 2^2, 3^3)$ & $6_{20}$ & $1_{1}$, $3_{15}$, $6_{15}$ & $1$ \\
$y^4(a+y^2) + x^3(1 + ay^2)$ & $D_6$, $(0; 2^2, 3, 6)$ & $12_{10}$ & $1_{1}$, $3_{3}$, $6_{13}$, $12_{4}$ & $1$ \\
$x^{3} y^{3} + y^{6} + a x^{3} + y^{3}$ & ${S_3 \times S_3}$, $(0; 2^3, 3)$ & $6_{2}$, $18_{6}$ & $1_{1}$, $3_{6}$, $9_{9}$, $36_{1}$ & $1$ \\
$x^{3} y^{3} + y^{6} - x^{3} + y^{3}$ & $(S_3 \times S_3) \rtimes C_2 {\cong (72, 40)}$, $(0; 2, 4, 6)$ & $12_{1}$, $36_{3}$ & $1_{1}$, $6_{3}$, $9_{3}$, $18_{3}$, $36_{1}$ & $1$ \\
$x^{2} y^{3} + y^{4} + a^5 x^{3} + x y$ & $D_5$, $(0; 2^2, 5^5)$ & $10_{12}$ & $1_{1}$, $5_{15}$, $10_{6}$ & $1$ \\
$xy - x^{3}+ y^{4} + x^{2} y^{3} $ & $S_5$, $(0; 2, 4, 5)$ & $20_{3}$, $60_{1}$ & $1_{1}$, $5_{3}$, $10_{3}$, $30_{3}$ & $1$ \\ \hline 
$ y^{2} - (x^{9} - 1)$ & $C_9$, $(0; 9^3)$ & $3_{1}$, $9_{13}$ & $1_{1}$, $9_{15}$ & $1$ \\
$ y^2 - x (x^4-1) (x^4 + 2i\sqrt{3}x^2 + 1)$ & $A_4$ & $4_3$, $6_4$, $12_7$ & $4_1$, $6_4$, $12_9$ & 0 \\
$y^2 -x(x^{8} - 1)$ & $D_8$ & $8_{5}$, $16_{5}$ & $2_{2}$, $4_{1}$, $8_{10}$, $16_{3}$ & $0$ \\
$y^2 - (x^{10} - 1)$ & $D_{10}$ & $10_{4}$, $20_{4}$ & $1_{1}$, $5_{3}$, $10_{8}$, $20_{2}$ & $1$
\label{tab: orbit decomposition, genus 4}
\end{longtable}


\begin{longtable}{p{0.37\linewidth}|p{0.16\linewidth}|p{0.16\linewidth}|p{0.16\linewidth}|p{0.04\linewidth}}
\caption{Orbit decomposition, three non-hyperelliptic genus-5 curves, and separately all hyperelliptic genus-5 curves with many automorphisms} \\
    $f$ & $\Autb$, $\bm{c}$ & Odd & Even & I\\ \hline \hline 
    $y^{4} -4 (x^{4} - ax^2 + 1) y^2 + b^2 x^{4} $ & $C_2^4$ & $4_{40}$, $8_{30}$, $16_{6}$ & $1_{32}$, $4_{40}$, $8_{30}$, $16_{6}$ & $32$ \\
    $y^3 - x^2(x^5-1)$ & $C_3 \times D_5$, $(0; 2, 6, 15)$ & $1_1$, $15_5$, $30_{14}$ & $3_1$, $15_{15}$, $30_{10}$  & $1$ \\
$4 x^{8} + 36 x^{4} y^{4} + 81 y^{8} + 8 x^{6} + 30 x^{2} y^{4} + 5 x^{4} + 14 y^{4} + 2 x^{2} + 1$ & $( ( {( C_4 \times C_2)} \rtimes {C_4 ) \rtimes C_3 )} \rtimes C_2$ $\cong (192, 181)$, $(0; 2, 3, 8)$ & $16_{1}$, $24_{2}$, $48_{1}$, $96_{4}$ & $1_{1}$, $3_{1}$, $4_{1}$, $6_{2}$, $12_{1}$, $16_{1}$, $24_{4}$, $48_{4}$, $96_{2}$ & $1$ \\ \hline
$y^{2} -(x^{11} - 1)$ & $C_{11}$, $(0; 11^3)$ & $1_{1}$, $11_{45}$ & $11_{48}$ & $1$ \\
$y^{2} - x(x^{10}- 1)$ & $D_{10}$ & $1_{1}$, $5_{3}$, $10_{12}$, $20_{18}$ & $1_{1}$, $2_{1}$, $5_{3}$, $10_{27}$, $20_{12}$ & $2$ \\
$y^2 - (x^{12} - 1)$ & $D_{12}$ & $1_{1}$, $3_{1}$, $6_{2}$, $12_{12}$, $24_{14}$ & $1_{1}$, $2_{1}$, $3_{1}$, $6_{5}$, $12_{25}$, $24_{8}$ & $2$ \\
$y^{2} - (x^{12} - 33 x^{8} - 33 x^{4} + 1)$ & $S_4$ & $1_{1}$, $3_{1}$, $6_{2}$, $12_{8}$, $24_{16}$ & $3_{2}$, $4_{1}$, $6_{3}$, $8_{1}$, $12_{13}$, $24_{14}$ & $1$ \\
$ y^2 - x(x^{10} + 11 x^{5} - 1)$ & $A_5$, $(0; 3^2,5)$ & $1_{1}$, $15_{1}$, $30_{6}$, $60_{5}$ & $6_{3}$, $10_{3}$, $15_{4}$, $30_{12}$, $60_{1}$ & $1$ \\
\label{tab: orbit decomposition, genus 5}
\end{longtable}

\begin{longtable}{p{0.37\linewidth}|p{0.16\linewidth}|p{0.16\linewidth}|p{0.16\linewidth}|p{0.04\linewidth}}
\caption{Orbit decomposition, some non-hyperelliptic genus-6 curves, and separately all hyperelliptic genus-6 curves with many automorphisms} \\
    $f$ & $\Autb$, $\bm{c}$ & Odd & Even & I\\ \hline \hline 
$L_5(x,y) + L_3(x,y) + L_1(x,y)$ & $C_2$ & $1_{96}$, $2_{960}$ & $1_{160}$, $2_{960}$ & $256$ \\ 
$x^{5} + a x^{2} y^{3} + b x^{3} y + y^{4} + c x y^{2} + d x^{2} + e y$ & $C_3$ & $1_{6}$, $3_{670}$ & $1_{10}$, $3_{690}$ & $16$ \\
$L_5(x,y) + L_1(x,y)$ & $C_4$ & $1_{16}$, $2_{40}$, $4_{480}$ & $2_{80}$, $4_{480}$ & $16$ \\
$x^{5} + a x y^{4} + b x^{2} y^{2} + c x^{3} + d y^{2} + ex$ & $C_4$ & $1_{8}$, $2_{44}$, $4_{480}$ & $1_{8}$, $2_{76}$, $4_{480}$ & $16$ \\
$1 + L_5(x, y)$ & $C_5$ & $1_1$, $5_{403}$ & $5_{416}$ & $1$ \\
$x^{5} +a x^{2} y^{3} +b x^{3} y + y^{4} +c x y^{2} + d x^{2} + y$ & $S_3$ & $1_{6}$, $3_{90}$, $6_{290}$ & $1_{10}$, $3_{150}$, $6_{270}$ & $16$ \\ 
$x^{5} + y^{4} + a x^{3} + x$ & $C_8$ & $1_{4}$, $2_{6}$, $4_{20}$, $8_{240}$ & $4_{40}$, $8_{240}$ & $4$ \\
$x^{5} + y^{5} + a x y^{3} + b x^{2} y + 1$ & $D_5$ & $1_{6}$, $5_{90}$, $10_{156}$ & $1_{10}$, $5_{150}$, $10_{132}$ & $16$ \\
$x^{5} + y^{5} + a x^{3} + x$ & $C_{10}$ & $1_{1}$, $5_{19}$, $10_{192}$ & $5_{32}$, $10_{192}$ & $1$ \\ 
$x^{5} + y^{4} + x$ & $C_{16}$ & $1_{2}$, $2_{1}$, $4_{3}$, $8_{10}$, $16_{120}$ & $8_{20}$, $16_{120}$ & $2$ \\
$x^{5} + y^{5} + x$ & $C_{20}$ & $1_{1}$, $5_{3}$, $10_{8}$, $20_{96}$ & $10_{16}$, $20_{96}$ & $1$ \\
$x^{5} + y^{4} + y$ & ${C_5 \times S_3}$, $(0; 2, 10, 15)$ & $1_{1}$, $5_{1}$, $15_{18}$, $30_{58}$ & $5_{2}$, $15_{30}$, $30_{54}$ & $1$ \\
$x^{4} y + y^{4} + x$ & ${C_{13} \rtimes C_3}$ ${\cong (39,1)}$, $(0; 3^2, 13)$ & $1_{1}$, $13_{5}$, $39_{50}$ & $13_{10}$, $39_{50}$ & $1$ \\
$y^3 - (x^4 - 1)^2(x^4+1)$ & $C_3 \rtimes D_8 \cong (48, 15)$, $(0; 2, 6, 8)$ & $24_4$, $48_{40}$ & $1_1$, $3_1$, $6_2$, $12_8$, $24_{34}$, $48_{24}$& $1$\\
$y^3 - (x^4 - 2i\sqrt{3}x^2 + 1)(x^4 + 2i\sqrt{3}x^2 + 1)^2$ & $(V_4 \rtimes C_9) \rtimes {C_2 \cong} (72, 15)$, $(0; 2, 4, 9)$ & $18_{4}$, $36_{6}$, $72_{24}$ & $1_{1}$, $9_{7}$, $18_{4}$, $36_{30}$, $72_{12}$ & $1$ \\
$(x^6 + y^6 + 1) + (x^2 + y^2 + 1)(x^4 + y^4 + 1) - 12x^2 y^2$ & $S_5$, $(0; 2, 4, 6)$ & $6_{2}$, $12_{2}$, $20_{3}$, $30_{2}$, $60_{17}$, $120_{7}$ & $1_{2}$, $2_{1}$, $10_{2}$, $12_{3}$, $15_{6}$, $20_{2}$, $30_{9}$, $60_{21}$, $120_{3}$ & $2$ \\
$x^{5} + y^{5} + 1$ & ${C_5^2 \rtimes S_3}$ ${\cong (150, 5)}$, $(0; 2, 3, 10)$ & $1_{1}$, $15_{1}$, $25_{5}$, $75_{13}$, $150_{6}$ & $15_{2}$, 
$25_{10}$, $75_{20}$, $150_{2}$ & $1$ \\ \hline
$y^2 - (x^{13} - 1)$ & $C_{13}$, $(0; 13^3)$ & $1_{1}$, $13_{155}$ & $13_{160}$ & $1$ \\
$y^2 - x(x^{12} - 1)$ & $D_{12}$ & $2_{1}$, $4_{1}$, $6_{1}$, $12_{17}$, $24_{75}$ & $2_{1}$, $4_{2}$, $6_{3}$, $12_{43}$, $24_{64}$ & $0$ \\
$y^{2} - x(x^4-1)(x^8 + 14x^4 + 1)$ & $S_4$ & $6_{4}$, $8_{3}$, $12_{6}$, $24_{79}$ & $4_{4}$, $6_{4}$, $8_{3}$, $12_{30}$, $24_{69}$ & $0$ \\
$y^2 - (x^{14} - 1)$ & $D_{14}$ & $14_{16}$, $28_{64}$ & $1_{1}$, $7_{7}$, $14_{41}$, $28_{52}$ & $1$ 
\label{tab: orbit decomposition, genus 6}
\end{longtable}


\begin{longtable}{p{0.37\linewidth}|p{0.16\linewidth}|p{0.16\linewidth}|p{0.16\linewidth}|p{0.04\linewidth}}
\caption{Orbit decomposition, some non-hyperelliptic genus-7 curves, and separately some hyperelliptic genus-7 curves including all with many automorphisms}\\
    $f$ & $\Autb$, $\bm{c}$ & Odd & Even & I \\ \hline \hline 
$(x^{3} + y^{3})^2 - x^{2} y^{2} - 1$ & $D_6$ & $1_{4}$, $2_{12}$, $3_{12}$, $6_{232}$, $12_{556}$ & $1_{12}$, $2_{12}$, $3_{36}$, $6_{336}$, $12_{508}$ & $16$ \\
$x^{6} + y^{6} - x^{3} - y^{3}$ & $C_3 \times S_3$ & $1_{1}$, $3_{1}$, $6_{7}$, $9_{20}$, $18_{439}$ & $3_{4}$, $6_{6}$, $9_{60}$, $18_{426}$ & $1$ \\
$x^{6} + y^{4} - 1$ & ${C_{12} \times C_2}$, $(0; 4, 6, 12)$ & $1_{1}$, $3_{1}$, $6_{6}$, $12_{42}$, $24_{316}$ & $1_{1}$, $2_{1}$, $3_{1}$, $6_{11}$, $12_{66}$, $24_{308}$ & $2$ \\
$(x^{4} + y^{4})^2 - x^{3} y^{3} - x^{2} y^{2}$ & ${C_8 \rtimes V_4}$ ${\cong (32, 43)}$, $(0; 2^3, 8)$ & $8_{16}$, $16_{84}$, $32_{208}$ & $2_{4}$, $4_{22}$, $8_{42}$, $16_{129}$, $32_{180}$ & $0$ \\
$x^{7} + y^{7} - x^{2} y^{2}$ & ${C_3 \times D_7}$ & $1_{1}$, $21_{21}$, $42_{183}$ & $3_{1}$, $21_{63}$, $42_{165}$ & $1$ \\
$y^{21} - x (x + 1)^{13} (x - 1)^7$ & ${C_3 \times D_7}$, $(0; 2, 6, 21)$ & $1_{1}$, $21_{21}$, $42_{183}$ & $3_{1}$, $21_{63}$, $42_{165}$ & $1$ \\
$x^{9} + y^{9} - x^{6}$ & ${C_3 \times D_9}$, $(0; 2, 6, 9)$ & $1_{1}$, $18_{3}$, $27_{21}$, $54_{139}$ & $3_{1}$, $18_{4}$, $27_{63}$, $54_{120}$ & $1$ \\
$x^{9} + y^{9} - x^{3} y^{3}$ & ${C_3 \times D_9}$, $(0; 2, 6, 9)$ & $1_{1}$, $18_{3}$, $27_{21}$, $54_{139}$ & $3_{1}$, $18_{4}$, $27_{63}$, $54_{120}$ & $1$ \\
$y^8 - (x^2 - 1) (x^2 + 1)^3$ & ${(C_{16} \rtimes C_2)} \rtimes C_2 {\cong (64, 41)}$, $(0; 2, 4, 16)$ & $16_{8}$, $32_{42}$, $64_{104}$ & $4_{2}$, $8_{9}$, $16_{15}$, $32_{64}$, $64_{92}$ & $0$ \\
$y^{16} - x (x - 1)^9 (x + 1)^6$ & ${(C_8 \rtimes C_4)} \rtimes C_2 {\cong (64, 41)}$, $(0; 2, 4, 16)$ & $16_{8}$, $32_{42}$, $64_{104}$ & $4_{2}$, $8_{9}$, $16_{15}$, $32_{64}$, $64_{92}$ & $0$ \\
$y^{9} - 6 y^{6} + 3(9 x^{4}-5) y^{3} - 8$ & $((C_4 \times S_3) \rtimes {C_2) \rtimes C_3}$ ${\cong (144, 127)}$, $(0; 2, 3, 12)$ & $4_{1}$, $12_{1}$, $24_{2}$, $36_{2}$, $72_{41}$, $144_{35}$ & $6_{2}$, $12_{2}$, $18_{6}$, $24_{2}$, $36_{28}$, $72_{28}$, $144_{35}$ & $0$ \\
$28 x^{4} y^{4} + 2 x^{7} + 2 y^{7} + 35 x^{3} y^{3} + 21 x^{2} y^{2} + 7 x y + 1$ & $\PSL_2(\mathbb{F}_8)$, $(0; 2, 3, 7)$ & $28_{1}$, $36_{1}$, $252_{14}$, $504_{9}$ & $28_{3}$, $36_{3}$, $126_{16}$, $252_{18}$, $504_{3}$ & $0$ \\ \hline
$y^2 - x(x^6-1)(x^8-2)$ & $C_2$ & $1_{64}$, $2_{4032}$ & $1_{64}$, $2_{4096}$ & $128$ \\
$y^{2} - x(x^7-1)(x^7-2)$ & $D_7$ & $1_{1}$, $7_{63}$, $14_{549}$ & $1_{3}$, $7_{189}$, $14_{495}$ & $4$ \\
$y^2 - (x^{15} - 1)$ & $C_{15}$, $(0; 15^3)$ & $3_{1}$, $5_{2}$, $15_{541}$ & $1_{1}$, $5_{1}$, $15_{550}$ & $1$ \\
$y^2 - (x^8-1)(x^8-2)$ & $D_8$ & $8_{72}$, $16_{472}$ & $1_{4}$, $2_{4}$, $4_{25}$, $8_{170}$, $16_{424}$ & $4$ \\
$y^2 - x(x^{14} - 1)$ & $D_{14}$ & $1_{1}$, $7_{7}$, $14_{57}$, $28_{260}$ & $1_{1}$, $2_{1}$, $7_{7}$, $14_{120}$, $28_{233}$ & $2$ \\
$y^2 -( x^{16} - 1)$ & $D_{16}$ & $8_{4}$, $16_{62}$, $32_{222}$ & $1_{2}$, $2_{1}$, $4_{3}$, $8_{14}$, $16_{112}$, $32_{198}$ & $2$ \\
$y^2 -(x^{16} + 1)$ & $D_{16}$ & $8_{4}$, $16_{62}$, $32_{222}$ & $1_{2}$, $2_{1}$, $4_{3}$, $8_{14}$, $16_{112}$, $32_{198}$ & $2$ 
\label{tab: orbit decomposition, genus 7}
\end{longtable}

\begin{longtable}
{p{0.37\linewidth}|p{0.16\linewidth}|p{0.16\linewidth}|p{0.16\linewidth}|p{0.04\linewidth}}
\caption{Orbit decomposition, all hyperelliptic curves of genus 8 with many automorphisms} \\ 
    $f$ & $\Autb$, $\bm{c}$ & Odd & Even & I\\ \hline \hline 
    $y^2 - (x^{17} - 1)$ & $C_{17}$, $(0; 17^3)$ & $17_{1920}$ & $1_{1}$, $17_{1935}$ & $1$ \\
$y^2 - x(x^4 - 1)(x^{12} - 33 x^8 - 33 x^4 + 1)$ & $S_4$ & $4_{2}$, $8_{2}$, $12_{94}$, $24_{1312}$ & $4_{2}$, $6_{8}$, $8_{4}$, $12_{90}$, $24_{1322}$ & $0$ \\
$y^2 - x (x^{16} - 1)$ & $D_{16}$ & $16_{72}$, $32_{984}$ & $2_{2}$, $4_{1}$, $8_{7}$, $16_{188}$, $32_{932}$ & $0$ \\
$y^2 -(x^{18} - 1)$ & $D_{18}$ & $6_{1}$, $18_{63}$, $36_{875}$ & $1_{1}$, $3_{1}$, $6_{1}$, $9_{14}$, $18_{182}$, $36_{819}$ & $1$
\label{tab: orbit decomposition, genus 8}
\end{longtable}


\begin{longtable}
{p{0.37\linewidth}|p{0.16\linewidth}|p{0.16\linewidth}|p{0.16\linewidth}|p{0.04\linewidth}}
\caption{Orbit decomposition, two non-hyperelliptic curves of genus 9, and all hyperelliptic curves of genus 9 with many automorphisms} \\
    $f$ & $\Autb$, $\bm{c}$ & Odd & Even & I\\ \hline \hline 
    $x^5 y^2 + y^5 + x^2$ & ${C_{19} \rtimes C_3} $ ${\cong (57,1)}$, $(0; 3^2, 19)$ & $1_1$, $19_{27}$, $57_{2286}$ & $19_{36}$, $57_{2292}$ & 1 \\ 
    $y - x^3 -x^4y^3 + xy^4 + 3x^2y^2$ & $S_5$, $(0; 2, 5, 6)$ & $6_1$, $10_4$, $20_{12}$, $30_{35}$, $60_{300}$, $120_{929}$ & $1_2$, $5_6$, $6_1$, $10_{10}$, $15_{24}$, $20_6$, $30_{85}$, $60_{402}$, $120_{867}$ & 2 \\ \hline
    $y^2 - (x^{19} - 1)$ & $C_{19}$, $(0; 19^3)$ & $1_1$, $19_{6885}$ & $19_{6912}$ & 1 \\
    $y^2 - (x^{12} - 33 x^8 - 33x^4 + 1)(x^8 + 14 x^4 + 1)$ & $S_4$ & $1_2$, $2_1$, $3_2$, $4_2$, $6_{13}$, $8_{11}$, $12_{172}$, $24_{5357}$ & $3_4$, $4_4$, $6_{18}$, $8_{16}$, $12_{290}$, $24_{5316}$ & 2 \\
    $y^2 - x(x^{18} - 1)$ & $D_{18}$ & $1_1$, $3_1$, $6_2$, $9_{14}$, $12_1$, $18_{245}$, $36_{3507}$ & $1_1$, $2_1$, $3_1$, $6_5$, $9_{14}$, $18_{497}$, $36_{3395}$ & 2 \\
    $y^2 - (x^{20} - 1)$ & $D_{20}$ & $1_1$, $5_3$, $10_{12}$, $20_{246}$, $40_{3144}$ & $1_1$, $2_1$, $5_3$, $10_{27}$, $20_{504}$, $40_{3024}$ & 2 \\
    $y^2 - (x^{20} - 228 x^{15} + 494 x^{10} + 228 x^5 + 1)$ & $A_5$, $(0; 3, 5^2)$ & $1_{1}$, $5_{3}$, $10_{3}$, $15_{4}$, $20_{9}$, $30_{117}$, $60_{2117}$ & $6_{3}$, $10_{12}$, $15_{16}$, $20_{12}$, $30_{345}$, $60_{2006}$ & $1$ 
\label{tab: orbit decomposition, genus 9}
\end{longtable}


\bibliography{jabref_library}

\providecommand{\bysame}{\leavevmode\hbox to3em{\hrulefill}\thinspace}
\providecommand{\MR}{\relax\ifhmode\unskip\space\fi MR }
\providecommand{\MRhref}[2]{%
  \href{http://www.ams.org/mathscinet-getitem?mr=#1}{#2}
}
\providecommand{\href}[2]{#2}
\begin{thebibliography}{{LMF}23}

\bibitem[Ati71]{Atiyah1971}
M.~F. Atiyah, \emph{Riemann surfaces and spin structures}, Annales
  scientifiques de l'\'Ecole Normale Sup\'erieure, Serie 4 \textbf{4} (1971),
  no.~1, 47--62.

\bibitem[Bar12]{Bars2012}
F.~Bars, \emph{On the automorphism group of genus 3 curves}, Surveys in
  Mathematics and Mathematical Sciences \textbf{2} (2012), no.~2, 83--124.

\bibitem[BB16]{Badr2016}
E.~Badr and F.~Bars, \emph{Automorphism groups of nonsingular plane curves of
  degree 5}, Communications in Algebra \textbf{44} (2016), no.~10, 4327--4340.

\bibitem[BDH22]{DisneyHogg2022b}
H.~W. Braden and L.~Disney-Hogg, \emph{Bring's curve: Old and new},
  arXiv:2208.13692, 2022.

\bibitem[Bea13]{Beauville2013}
Arnaud Beauville, \emph{Vanishing thetanulls on curves with involutions},
  Rendiconti del Circolo Matematico di Palermo \textbf{62} (2013), no.~1,
  61--66.

\bibitem[BF02]{Bailey2002}
Paul Bailey and Michael~D. Fried, \emph{Hurwitz monodromy, spin separation and
  higher levels of a modular tower.}, Arithmetic fundamental groups and
  noncommutative algebra. Proceedings of the 1999 von Neumann conference,
  Berkeley, CA, USA, August 16--27, 1999, Providence, RI: American Mathematical
  Society (AMS), 2002, pp.~79--220 (English).

\bibitem[BGS07]{Biswas2007}
I.~Biswas, S.~Gadgil, and P.~Sankaran, \emph{On theta characteristics of a
  compact {R}iemann surface}, Bulletin des Sciences Mathématiques \textbf{131}
  (2007), no.~5, 493--499.

\bibitem[BN12]{Braden2012}
H.~W. Braden and T.~P. Northover, \emph{{B}ring’s curve: its period matrix
  and the vector of {R}iemann constants}, Symmetry, Integrability and Geometry:
  Methods and Applications \textbf{8} (2012), 065.

\bibitem[Bol87]{Bolza1887}
O.~Bolza, \emph{On binary sextics with linear transformations into themselves},
  American Journal of Mathematics \textbf{10} (1887), no.~1, 47--70.

\bibitem[Bre00]{Breuer2000}
T.~Breuer, \emph{Characters and automorphism groups of compact {R}iemann
  surfaces}, London Mathematical Society Lecture Note Series, vol. 280,
  Cambridge University Press, 2000.

\bibitem[Bro87]{Broughton1987}
S.~A. Broughton, \emph{The homology and higher representations of the
  automorphism group of a {R}iemann surface}, Transactions of the American
  Mathematical Society \textbf{300} (1987), no.~1, 153--158.

\bibitem[Bro90]{Broughton1990}
\bysame, \emph{The equisymmetric stratification of the moduli space and the
  {K}rull dimension of mapping class groups}, Topology and its Applications
  \textbf{37} (1990), no.~2, 101--113.

\bibitem[Bro91]{Broughton1991}
\bysame, \emph{Classifying finite group actions on surfaces of low genus},
  Journal of Pure and Applied Algebra \textbf{69} (1991), no.~3, 233--270.

\bibitem[BRR13]{Behn2013}
A.~Behn, R.~E. Rodríguez, and A.~M. Rojas, \emph{Adapted hyperbolic polygons
  and symplectic representations for group actions on {R}iemann surfaces},
  Journal of Pure and Applied Algebra \textbf{217} (2013), no.~3, 409--426.

\bibitem[BSZ19]{Bruin2019}
N.~Bruin, J.~Sijsling, and A.~Zotine, \emph{Numerical computation of
  endomorphism rings of {J}acobians}, The Open Book Series \textbf{2} (2019),
  no.~1, 155--171.

\bibitem[Con80]{Conder1980}
M.~D.~E. Conder, \emph{Generators for alternating and symmetric groups},
  Journal of the London Mathematical Society \textbf{s2-22} (1980), no.~1,
  75--86.

\bibitem[Con84]{Conder1984}
\bysame, \emph{Some results on quotients of triangle groups}, Bulletin of the
  Australian Mathematical Society \textbf{30} (1984), no.~1, 73--90.

\bibitem[Con85]{Conway1985}
John~H. Conway, \emph{Atlas of finite groups: maximal subgroups and ordinary
  characters for simple groups.}, Clarendon, Oxford, 1985.

\bibitem[Con87]{Conder1987}
M.~Conder, \emph{The genus of compact {R}iemann surfaces with maximal
  automorphism group}, Journal of Algebra \textbf{108} (1987), no.~1, 204--247.

\bibitem[Con90]{Conder1990}
\bysame, \emph{Hurwitz groups: A brief survey}, Bulletin (New Series) of the
  American Mathematical Society \textbf{23} (1990), no.~2, 359--370.

\bibitem[DA03]{Darafsheh2003}
M.~R. Darafsheh and A.~R. Ashrafi, \emph{Generating pairs for the sporadic
  group $ru$}, Journal of Applied Mathematics and Computing \textbf{12} (2003),
  no.~1, 143--154.

\bibitem[DH23]{DisneyHogg2023}
L.~Disney-Hogg, \emph{Symmetries of {R}iemann surfaces and magnetic monopoles},
  Ph.D. thesis, University of Edinburgh, 2023.

\bibitem[Dol97a]{Dolgachev1997}
I.~Dolgachev, \emph{Lectures on modular forms}, 1997, Accessed May 2023.

\bibitem[Dol97b]{Dolgachev1997a}
I.~V. Dolgachev, \emph{Invariant stable bundles over modular curves $x(p)$},
  arxiv:alg-geom/9710012, 1997.

\bibitem[Edg67]{Edge1967}
W.~L. Edge, \emph{A canonical curve of genus 7}, Proceedings of the London
  Mathematical Society \textbf{s3-17} (1967), no.~2, 207--225.

\bibitem[Edg81]{Edge1981a}
\bysame, \emph{A pencil of four-nodal plane sextics}, Mathematical Proceedings
  of the Cambridge Philosophical Society \textbf{89} (1981), no.~3, 413--421.

\bibitem[Edg84]{Edge1984}
\bysame, \emph{Fricke's octavic curve}, Proceedings of the Edinburgh
  Mathematical Society \textbf{27} (1984), no.~1, 91--101.

\bibitem[EGM05]{EtayoGordejuela2005}
J.~J. Etayo~Gordejuela and E.~Martínez, \emph{Alternating groups, {H}urwitz
  groups and ${H^\ast}$-groups}, Journal of Algebra \textbf{283} (2005), no.~1,
  327--349.

\bibitem[Far12]{Farkas2012}
G.~Farkas, \emph{Theta characteristics and their moduli}, Milan Journal of
  Mathematics \textbf{80} (2012), no.~1, 1--24.

\bibitem[Fay73]{Fay1973}
J.~D. Fay, \emph{Theta functions on {R}iemann surfaces}, Lecture Notes in
  Mathematics, vol. 352, Springer, 1973.

\bibitem[Fin47]{Fine1947}
N.~J. Fine, \emph{Binomial coefficients modulo a prime}, The American
  Mathematical Monthly \textbf{54} (1947), no.~10, 589--592.

\bibitem[FR73]{Finkelstein1973}
L.~Finkelstein and A.~Rudvalis, \emph{Maximal subgroups of the
  {H}all-{J}anko-{W}ales group}, Journal of Algebra \textbf{24} (1973), no.~3,
  486--493.

\bibitem[Fri99]{Fricke1899}
R.~Fricke, \emph{Ueber eine einfache gruppe von 504 operationen}, Mathematische
  Annalen \textbf{52} (1899), no.~2, 321--339.

\bibitem[FV89]{Fried1989}
M~Fried and Helmut V{\"o}lklein, \emph{Unramified abelian extensions of
  {G}alois covers}, Proceedings of Symposia in Pure Mathematics, vol.~49, 1989,
  Part I, pp.~675--693.

\bibitem[GAP22]{GAP4}
The GAP~Group, \emph{{GAP -- Groups, Algorithms, and Programming, Version
  4.12.2}}, 2022.

\bibitem[Gro57]{Grothendieck1957}
Alexander Grothendieck, \emph{Sur quelques points d'algèbre homologique},
  Tohoku Mathematical Journal \textbf{9} (1957), no.~2, 119--221.

\bibitem[Har66]{Harvey1966}
W.~J. Harvey, \emph{Cyclic groups of automorphisms of a compact {R}iemann
  surface}, Q J Math \textbf{17} (1966), no.~1, 86--97.

\bibitem[Hid17]{Hidalgo2017}
R.~A. Hidalgo, \emph{About the {F}ricke-{M}acbeath curve}, arXiv:1703.01869,
  2017.

\bibitem[Igu72]{Igusa1972}
J.~I. Igusa, \emph{Theta functions}, Die Grundlehren der mathematischen
  Wissenschaften in Einzeldarstellungen mit besonderer Berucksichtigung der
  Anwendungsgebiete, vol. 194, Springer, 1972.

\bibitem[Joh80]{Johnson1980}
D.~Johnson, \emph{Spin structures and quadratic forms on surfaces}, Journal of
  the London Mathematical Society \textbf{s2-22} (1980), no.~2, 365--373.

\bibitem[KR89]{Kani1989}
E.~Kani and M.~Rosen, \emph{Idempotent relations and factors of {J}acobians},
  Mathematische Annalen \textbf{284} (1989), no.~2, 307--327.

\bibitem[KS10]{Kallel2010}
S.~Kallel and D.~Sjerve, \emph{Invariant spin structures on {R}iemann
  surfaces}, Annales de la Facult{\'e} des Sciences de Toulouse \textbf{19}
  (2010), 457--477.

\bibitem[Lef21a]{Lefschetz1921}
S.~Lefschetz, \emph{On certain numerical invariants of algebraic varieties with
  application to abelian varieties}, Transactions of the American Mathematical
  Society \textbf{22} (1921), no.~3, 327--406.

\bibitem[Lef21b]{Lefschetz1921a}
\bysame, \emph{On certain numerical invariants of algebraic varieties with
  application to abelian varieties}, Transactions of the American Mathematical
  Society \textbf{22} (1921), no.~4, 407--482.

\bibitem[{LMF}23]{LMFDB}
The {LMFDB Collaboration}, \emph{The {L}-functions and modular forms database},
  \url{http://www.lmfdb.org}, 2023.

\bibitem[Mac65]{Macbeath1965}
A.~M. Macbeath, \emph{On a curve of genus 7}, Proceedings of the London
  Mathematical Society \textbf{s3-15} (1965), no.~1, 527--542.

\bibitem[{Map}22]{Maple2022}
{Maplesoft, a division of Waterloo Maple Inc.}, \emph{Maple}, 2022.

\bibitem[MP21]{Muller2021}
N.~M{\"{u}}ller and R.~Pink, \emph{Hyperelliptic curves with many
  automorphisms}, International Journal of Number Theory \textbf{18} (2021),
  no.~4, 913--930.

\bibitem[MSSV02]{Magaard2002}
K.~Magaard, T.~Shaska, S.~Shpectorov, and H.~Voelklein, \emph{The locus of
  curves with prescribed automorphism group}, Communications on Arithmetic
  Fundamential Groups and Galois Theory, Kyoto Technical Report Series, 2002.

\bibitem[Pop72]{Popp1972}
H.~Popp, \emph{On a conjecture of {H}. {R}auch on theta constants and {R}iemann
  surfaces with many automorphisms.}, Journal f{\"{u}}r die reine und
  angewandte Mathematik \textbf{1972} (1972), no.~253, 66--77.

\bibitem[Rau70]{Rauch1970}
H.~E. Rauch, \emph{Theta constants on a {R}iemann surface with many
  automorphisms}, Symposia Mathematica \textbf{3} (1970), 305--323.

\bibitem[RCR22]{ReyesCarocca2022}
S.~Reyes-Carocca and A.~M. Rojas, \emph{On large prime actions on {R}iemann
  surfaces}, Journal of Group Theory \textbf{25} (2022), no.~5, 887--940.

\bibitem[Rie92]{Ries1992}
John F.~X. Ries, \emph{Splittable {Jacobi} varieties}, Curves, Jacobians, and
  Abelian varieties. Proceedings of an AMS-IMS-SIAM joint summer research
  conference on the Schottky problem, held June 21- 27, 1990 at the University
  of Massachusetts, Amherst, MA, USA, American Mathematical Society, 1992,
  pp.~305--326 (English).

\bibitem[Rie93]{Ries1993}
\bysame, \emph{Subvarieties of moduli space determined by finite groups acting
  on surfaces}, Transactions of the American Mathematical Society \textbf{335}
  (1993), no.~1, 385--406.

\bibitem[Sco77]{Scott1977}
Leonard~L. Scott, \emph{Matrices and cohomology}, Annals of Mathematics
  \textbf{105} (1977), no.~3, 473--492.

\bibitem[Ser90]{Serre1990}
Jean-Pierre Serre, \emph{Coverings with odd ramification and
  theta-characteristics}, C. R. Acad. Sci., Paris, S{\'e}r. I \textbf{311}
  (1990), no.~9, 547--552 (French).

\bibitem[Sha07]{Shaska2007}
T.~Shaska, \emph{Some open problems in computational algebraic geometry},
  Albanian Journal of Mathematics \textbf{1} (2007), no.~4, 297--319.

\bibitem[Shi95]{Shimura1995}
M.~Shimura, \emph{Defining equations of modular curves {$X_0(N)$}}, Tokyo
  Journal of Mathematics \textbf{18} (1995), no.~2, 443--456.

\bibitem[Swi16]{Swinarski2016}
D.~Swinarski, \emph{Equations of {R}iemann surfaces with automorphisms},
  arXiv:1607.04778, 2016.

\bibitem[Wei95]{Weibel1995}
C.~A. Weibel, \emph{An introduction to homological algebra}, Cambridge Studies
  in Advanced Mathematics, vol.~38, Cambridge University Press, 1995.

\bibitem[Wil93]{Wilson1993}
R.~A. Wilson, \emph{The symmetric genus of the baby monster}, Q J Math
  \textbf{44} (1993), no.~4, 513--516.

\bibitem[Wil01]{Wilson2001}
\bysame, \emph{The monster is a hurwitz group}, Journal of Group Theory
  \textbf{4} (2001), no.~4, 367--374.

\bibitem[Wim95a]{Wiman1895a}
A.~Wiman, \emph{{\"{U}}ber die hyperelliptischen curven und diejenigen vom
  geschlechte p= 3, welche eindeutigen transformationen in such zulassen},
  Bihang till Kongl. Svenska vetenskaps-akademiens handlingar \textbf{21}
  (1895), 1--23.

\bibitem[Wim95b]{Wiman1895}
\bysame, \emph{Ueber die algebraischen curven von den geschlechten p= 4, 5 und
  6, welche eindeutigen transformationen in sich besitzen}, Bihang till Kongl.
  Svenska vetenskaps-akademiens handlingar \textbf{21} (1895), 1--41, English
  translation available at https://arxiv.org/abs/2204.01656.

\bibitem[Zom10]{Zomorrodian2010}
R.~Zomorrodian, \emph{Group action on genus 7 curves and their {W}eierstrass
  points}, The Quaterly Journal of Mathematics \textbf{61} (2010), no.~4,
  511--540.

\end{thebibliography}
\addcontentsline{toc}{chapter}{Bibliography}
\bibliographystyle{amsalpha}

\end{document}